\documentclass[reqno, english,12pt]{amsart}
\RequirePackage{mathrsfs}
\usepackage{a4wide}

\usepackage[all]{xy}
\usepackage{amssymb,amsmath,amsfonts, amsthm, mathtools}
\usepackage{hyperref}
\usepackage{dsfont}
\usepackage{upgreek}
\usepackage{mathrsfs}
\usepackage{mathabx,yfonts}
\usepackage{enumerate}
\usepackage{cases}

\usepackage{enumitem}

\usepackage{graphicx}

\usepackage[OT2,T1]{fontenc}

\DeclareMathOperator{\chr}{char}

\DeclareMathOperator{\Gal}{Gal} 
 
\DeclareMathOperator{\Frob}{Frob} 
  
\DeclareMathOperator{\Spec}{Spec}

\newtheorem{theorem}{Theorem}

\newtheorem{lemma}[theorem]{Lemma}
\newtheorem{proposition}[theorem]{Proposition}

\theoremstyle{definition}
\newtheorem{definition}[theorem]{Definition}
\newtheorem{example}[theorem]{Example}
\newtheorem{remark}[theorem]{Remark}

\newtheorem*{notation}{Notation}

\numberwithin{theorem}{section}
\numberwithin{equation}{section}

\DeclareSymbolFont{bbold}{U}{bbold}{m}{n}
\DeclareSymbolFontAlphabet{\mathbbold}{bbold}

\renewcommand{\P}{\mathbb{P}}
\newcommand{\Q}{\mathbb{Q}}
\newcommand{\F}{\mathbb{F}}
\newcommand{\N}{\mathbb{N}}
\newcommand{\R}{\mathbb{R}}

\newcommand{\Z}{\mathbb{Z}}

\renewcommand{\l}{\left}
\newcommand{\m}{\mathfrak{m}}
\renewcommand{\r}{\right}

\renewcommand{\c}{\mathcal} 
\renewcommand{\gcd}{\textrm{gcd}} 
\renewcommand{\leq}{\leqslant}
\renewcommand{\geq}{\geqslant}
\renewcommand{\#}{\sharp}
\renewcommand{\gg}{\ggg}

\renewcommand{\ll}{\lll}

\newcommand{\p}{\mathfrak{p}}
\newcommand{\Adele}{\mathbf{A}}

\newcommand{\fp}{\mathfrak{p}}

\newcommand{\x}{\mathbf{x}}
\newcommand{\y}{\mathbf{y}}

\newcommand\FF{\mathbb{F}}
\newcommand\PP{\mathbb{P}}
\newcommand\ZZ{\mathbb{Z}}
\newcommand\NN{\mathbb{N}}
\newcommand\QQ{\mathbb{Q}}
\newcommand\RR{\mathbb{R}}

\usepackage{color}

\begin{document}
\begin{abstract}
 We study probability distributions arising from local obstructions to the existence of $p$-adic points in families of varieties. In certain cases we show that an Erd\H{o}s--Kac type normal distribution law holds.
\end{abstract}

\date{\today}

\title[
An Erd\H{o}s--Kac law for local solubility in families of varieties
]
{
An Erd\H{o}s--Kac law for local solubility in families of varieties
} 

\author{D. Loughran}
\address{
University of Manchester\\
School of Mathematics\\
Oxford Road\\
Manchester\\
M13 9PL\\
UK}
\email{daniel.loughran@manchester.ac.uk}

\author{E. Sofos}
\address{
Max Planck Institute for Mathematics\\
Vivatsgasse 7
\\
Bonn
\\
53111
\\
Germany}
\email{sofos@mpim-bonn.mpg.de}


\maketitle

\tableofcontents

\section{Introduction}
\label{s:intro} 

\subsection{A central limit theorem for fibrations} \label{sec:CLT}

Let $V$ be a smooth projective variety over $\QQ$ equipped with a dominant morphism $\pi: V \to \PP^n$
with geometrically integral generic fibre. We view $\pi$ as defining a family of varieties given by the fibres of $\pi$.

A natural problem is to study the distribution of the varieties in the family with a rational point.
In the case of families of conics, this problem was studied by Serre in \cite{Ser90}.
He obtained precise upper bounds for the counting function
$$N(\pi,B) := \#\{ x \in \PP^n(\QQ) : x \in \pi(V(\QQ)), H(x) \leq B\}$$
of the number of varieties in the family with a rational point.
(Here $H$ is the usual naive height on $\PP^n(\QQ)$).
For example, his results show that for families of conics if $\pi$ admits no section 
over $\Q$
then $N(\pi,B) = o(B^{n+1})$, i.e.~$100\%$ of the fibres of $\pi$ have no rational point. 
He did this by showing that $100\%$ of the fibres fail to be $p$-adically soluble 
for all primes $p$.

This subject has been studied in various
settings by different authors; 
the reader is referred to~\cite{LS16} and \cite{Lou13}
for a history of the subject.
The main result from~\cite{LS16} generalised Serre's result from families of conics to arbitrary
families of varieties $\pi: V \to \PP^n$ for any smooth projective variety $V$. 
Again the authors considered the closely related problem of counting the number of varieties
in the family which are everywhere locally soluble. They proved
an upper bound of the shape
\begin{equation} \label{eqn:LS16}
 \#\{ x \in \PP^n(\QQ) : x \in \pi(V(\Adele_\QQ)), H(x) \leq B\} \ll \frac{B^{n+1}}{(\log B)^{\Delta(\pi)}},
\end{equation}
for an explicit non-negative
$\Delta(\pi) \in \QQ$. (Here $\Adele_\QQ$ denotes the adeles of $\QQ$.)
Moreover, they conjectured in \cite[Conj.~1.6]{LS16} that the upper bound \eqref{eqn:LS16} is sharp, under the necessary assumptions that the set being counted is non-empty and that the fibre over every codimension $1$ point of $\PP^n$ contains an irreducible component of multiplicity $1$. As it will occur frequently in our results, we recall the definition of $\Delta(\pi)$ here.
\begin{definition} \label{def:Delta}
	Let $\pi:V \to X$ be a dominant proper morphism of smooth irreducible varieties over a
	field $k$ of characteristic $0$. For each (scheme-theoretic) point
	$x \in X$ with residue field $\kappa(x)$,
	the absolute Galois group $\Gal(\overline{\kappa(x)}/ \kappa(x))$ 
	of the residue field acts on the irreducible
	components of $\pi^{-1}(x)_{\overline{\kappa(x)}}:=\pi^{-1}(x) \times_{\kappa(x)} \overline{\kappa(x)}$ of multiplicity $1$. 
	We choose
	some finite group $\Gamma_x$ through which this action factors. Then we define
	$$\delta_x(\pi) = \frac{\# \left\{ \gamma \in \Gamma_x : 
	\begin{array}{l}
		\gamma \text{ fixes an irreducible component} \\
		\text{of $\pi^{-1}(x)_{\overline{\kappa(x)}}$ of multiplicity } 1
	\end{array}
	\right \}}
	{\# \Gamma_x }
	\  \text{ and } \ 
	\Delta(\pi) = \sum_{D \in X^{(1)}} ( 1 - \delta_D(\pi)),$$
	where $X^{(1)}$ denotes the set of codimension $1$ points of $X$.
\end{definition}
These invariants  are defined by group theoretic data which can often be  calculated in practice. In this paper we consider the following problem which is closely related to Serre's:

\medskip
Given a family of varieties $\pi: V \to \PP^n$
and $j \in \Z_{\geq 0}$,
what is the distribution of varieties in the family which fail to have a $p$-adic point for exactly $j$ 
primes $p$? 
\medskip

To study this problem, for $x \in \PP^n(\QQ)$ we consider the function
\begin{equation} \label{def:omega}
	\omega_\pi(x):=\#\big\{\text{primes } p:\pi^{-1}(x)(\Q_p)=\emptyset\big\}.
\end{equation}
Note that $\omega_\pi(x)$ need not be finite in general; however it is finite if $\pi^{-1}(x)$ is geometrically integral, as follows from  the Lang--Weil estimates \cite{LW} and Hensel's lemma. If the generic fibre of $\pi$ is geometrically integral, then $\omega_\pi(x)$ is finite for all $x$ outside of some proper Zariski closed set. (In practice, we restrict  to the smooth fibres of $\pi$.) One can also consider variants of the function $\omega_\pi$ from \eqref{def:omega}, by considering
real solubility or by dropping conditions at finitely many primes; we discuss this possibility in \S \ref{sec:generalisations}. 

As is clear from \eqref{eqn:LS16}, if $\Delta(\pi) > 0$ then the function $\omega_\pi(x)$
is almost always positive. Our first result gives more specific information about the distribution of $\omega_\pi(x)$, and is an analogue of the Erd\H{o}s--Kac theorem \cite{MR0002374} in our setting. Recall that this states that  the function
\begin{equation} \label{def:omega_classic}
	\omega(m):=\#\{ \text{primes } p: p \mid m\}
\end{equation}
behaves likes a normal distribution with mean and variance $\log \log n$; more formally,
for every interval $\c{J}\subset \R$
one has 
\[
\lim_{x\to\infty}
\frac{1}{x}
\#\Big\{1\leq m \leq x:
\frac{\omega(m)-\log \log m}{\sqrt{\log \log  m}}
\in \mathcal{J}
\Big\}
=\frac{1}{\sqrt{2\pi}}
\int_{\c{J}}
\mathrm{e}^{-\frac{t^2}{\!2}}
\mathrm{d}t
.\]
This theorem is one of the foundational results in probabilistic number theory.

For our analogue, we need some notation.
For each $B\in \R_{\geq 1}$ 
and $A\subseteq \P^n(\Q)$ we define 
\begin{equation}\label{eq:puppenkonigEPI}
\nu_B(A):=
\frac{\#\big\{
x\in \P^n(\Q):H(x)\leq B, x \in A
\big\}}{\#\{x\in \P^n(\Q):H(x)\leq B\}}
.\end{equation}
If 
$\lim_{B\to\infty}\nu_B(A)$ exists then its value is to be conceived as
the ``density'' of $A$.
 Our result is  the following. (Here, and in what follows, we also commit the minor abuse of implicitly excluding the finitely many
rational points $x$ with $\log H(x) \leq 1$.)

\begin{theorem}
\label{thm:gaussian} 
Let $V$ be a smooth  projective variety over $\QQ$ equipped with a dominant morphism $\pi: V \to \PP^n$
with geometrically integral generic fibre and $\Delta(\pi)\neq 0$. 
Let $H$ be the usual naive height on $\P^n$. Then for any interval $\mathcal{J}\subset \R$ we have 
\[
\lim_{B\to\infty}
\nu_B
\left(
\left\{x \in \PP^n(\Q):
\frac{\omega_\pi(x)-\Delta(\pi)\log \log H(x)}{\sqrt{\Delta(\pi)\log \log  H(x)}}
\in \mathcal{J}
\right\}\right)
=\frac{1}{\sqrt{2\pi}}
\int_{\c{J}}
\mathrm{e}^{-\frac{t^2}{\!2}}
\mathrm{d}t
.\]
\end{theorem} 
Note that the probability distribution obtained only depends on the invariant $\Delta(\pi)$ from Definition \ref{def:Delta}; the geometric properties of the smooth members of the family are irrelevant.
A measure-theoretic interpretation of Theorem \ref{thm:gaussian} is as follows:
It says that 
$$\mathcal{J} \mapsto \lim_{B\to\infty}
\nu_B
\Bigg(
x;
\frac{\omega_\pi(x)-\Delta(\pi)\log \log H(x)}{\sqrt{\Delta(\pi)\log \log  H(x)}}\in \mathcal{J}\Bigg)
$$
defines a probability measure on $\RR$ which equals the standard Gaussian measure.
Informally, it says $\omega_\pi(x)$ is normally distributed with mean and variance
$\Delta(\pi) \log \log H(x)$.

Theorem~\ref{thm:gaussian} is proved by studying the moments
\begin{equation}
\label{def:momer}
\c{M}_r(\pi, B):=
\sum_{\substack{ x \in \PP^n(\QQ), H(x)\leq B\\ \pi^{-1}(x) \text{ smooth}}}
\l(
\frac{
\omega_\pi(x)
-
\Delta(\pi)\log \log B
}
{\sqrt{\Delta(\pi)\log \log B}}
\r)^{r}\!\!\!, \quad
(r \in \Z_{\geq 0}).
\end{equation} 

\begin{theorem}
\label{thm:moment}
Keep the assumptions of Theorem \ref{thm:gaussian}.
Then for each $r \in \Z_{\geq 0}$
we have
\[\frac{ \c{M}_r(\pi,B) } {\#\{x\in \P^n(\Q):H(x)\leq B\} } =\mu_r+O_{r}\Big(\frac{\log \log \log \log B}{(\log \log B)^{1/2}}\Big)
, \quad \text{where }
\mu_r:= \begin{cases} \frac{r!}{2^{r/2}(r/2)!}, & r \text{ even, }\\  0, & r \text{ odd.} \end{cases} \]
\end{theorem}
Here $\mu_r$ is the $r$-th moment of the standard normal distribution. 
 Our main tool in the proof of Theorem \ref{thm:moment}
is the result of Granville
and Soundararajan~\cite{MR2290492}.
Theorem \ref{thm:gaussian} is proved from Theorem \ref{thm:moment}  via a standard argument, which rests on the fact that the normal distribution is determined by its moments

There are general conditions under which 
one can prove an Erd\H{o}s--Kac law for certain additive functions defined on  $\Z^{n+1}$, 
see~\cite[\S 12]{MR560507} for example.
In principle,
these results
could be extended to 
cover additive arithmetic
functions restricted to values
of a general
polynomial, i.e.~$
\sum_{p|f(x)}h(p)
$
for integer polynomials $f$ and functions $h$ of certain 
growth over the primes;
see the work of Xiong~\cite{MR2497471}.
However, $\omega_\pi$ does not admit any
such interpretation, as the following example shows.

\begin{example}
Consider the following family of conics
$$ax^2 + by^2 = cz^2 \qquad \subset \PP^2 \times \PP^2,$$
equipped with the projection $\pi$ to $(a:b:c)$.
Take $(a,b,c) \in \ZZ^3$ pairwise coprime, square-free and all congruent to $1 \bmod 4$.
A Hilbert symbol calculation shows that
\begin{equation}
\label{eq:con}
\omega_\pi(a:b:c)=
\Bigg(\frac{1}{2}\sum_{p\mid a}
\Big(1-
\Big(\frac{bc}{p}\Big)
\Big) \Bigg)
+
\Bigg(\frac{1}{2}\sum_{p\mid b}
\Big(1-
\Big(\frac{ac}{p}\Big)
\Big) \Bigg)
+
\Bigg(\frac{1}{2}\sum_{p\mid c}
\Big(1-
\Big(\frac{-ab}{p}\Big)
\Big) \Bigg) 
,\end{equation}
where $(\frac{\cdot}{p})$ is the Legendre symbol (cf.~\cite[p.~13]{MR1199934}).
One cannot 
directly
apply the aforementioned general results here,
since the function in~\eqref{eq:con} is not the restriction of an additive function to the values of a polynomial. 
Nevertheless Theorem \ref{thm:gaussian} implies that the function $\omega_\pi$ has normal order $\frac{3}{2}\log \log H(a:b:c)$ in this case.
\end{example}

We also give an application of our results to a family of curves of genus $1$.

\begin{example}
	Let $c,d \in \ZZ$ be such that $cd(c-d) \neq 0$ and let $f(t) \in \ZZ[t]$ be a 
	square-free polynomial of even degree.
	Consider the variety
	$$W: \quad x^2 - cw^2 = f(t)y^2, \quad x^2 - dw^2 = f(t)z^2 \qquad \subset \PP^2 \times \mathbb{A}^{1}.$$
	Let $\pi: V \to \PP^1$ be a non-singular 
	compactification of the natural projection $W \to \PP^1$ to the $t$-coordinate.
	The generic fibre of $\pi$ is a smooth intersection of two quadrics in $\PP^3$,
	hence is a genus $1$ curve.
	The singular fibres occur over the closed points corresponding to the
	irreducible polynomials dividing $f$. Moreover,
	by \cite[Prop.~4.1]{CTSSD97}, the fibre over every such closed point $P$ is a double 
	fibre, hence $\delta_P(\pi) = 0$. Theorem \ref{thm:gaussian} therefore
	implies that $\omega_\pi$ has normal order $r(f)\log \log H(1:t)$
	in this case, where $r(f)$ is the number of irreducible polynomials dividing $f$.
\end{example}
This last example is particularly interesting, as the upper bound \eqref{eqn:LS16} is conjecturally sharp
only if the fibre over every codimension $1$ point contains an irreducible component of multiplicitly $1$. No such assumptions are required in the statements of our theorems.

The next example illustrates how to (essentially) recover the usual $\omega$ \eqref{def:omega_classic}
as a special case of our $\omega_\pi$.
\begin{example}
	Let $V$ be a smooth projective variety over $\QQ$ equipped with a dominant morphism
	$\pi:V \to \PP^1$ such that:
	\begin{enumerate}
		\item The fibre over $(0:1)$ has multiplicity $m > 1$,
		i.e.~we have $\pi^*((0:1)) = mD$
		for some divisor $D$ on $V$.
		\item All other fibres are geometrically integral.
	\end{enumerate}
	Examples of such varieties are ``unnodal Halphen surfaces of index $m$'' \cite[\S2]{CD12}.

	Let now $(x_0,x_1)$ be a primitive integer vector and $P=(x_0:x_1) \in \PP^1(\QQ)$.
	Then our methods will yield the existence of some $A > 0$
	such that for all primes $p > A$ we have
	$$
		v_p(x_0) = 0 \implies \pi^{-1}(P)(\QQ_p) \neq \emptyset, \quad
		v_p(x_0) = 1 \implies \pi^{-1}(P)(\QQ_p) =\emptyset,
	$$
	where $v_p$ denotes the $p$-adic valuation.
	Thus if $x_0$ is square-free and $p \nmid x_0$ for all $p \leq A$, then
	$\omega_{\pi}(P) = \omega(x_0).$
	(We shall see that small primes and primes of 
	higher multiplicity do not effect the overall probabilistic behaviour,
	so our results essentially recover
	the usual Erd\H{o}s--Kac theorem.)
\end{example}

\subsection{The pseudo-split case} \label{sec:ps_intro}
Our results from \S\ref{sec:CLT} only apply when $\Delta(\pi) \neq 0$. 
It turns out that
a normal distribution does not hold when $\Delta(\pi) = 0$. 
We refer to the case $\Delta(\pi) = 0$ as the ``pseudo-split case''. This is because
the condition $\Delta(\pi) = 0$ is equivalent to the condition that the fibre over every
codimension $1$ point of $\PP^n$ is \emph{pseudo-split}, in the sense of \cite[Def.~1.3]{LSS17}.
The pseudo-split case is interesting from an arithmetic perspective,
as these are exactly the families of
varieties for which
a positive proportion of the fibres can be
everywhere locally soluble (see \cite[Thm.~1.3]{LS16}).

In the pseudo-split case there is a discrete probability distribution,
in a sense that is made precise in the following theorem. For
$j\in \Z_{\geq 0}$
and $B\geq 1$ we define
\begin{equation}
\label{def:fpijb}
\tau_\pi(j,B)
:=
\frac
{\#\{
x\in \P^n(\Q):
H(x)\leq B,
\pi^{-1}(x) \text{ smooth},
\omega_\pi(x)=j
\}}
{\#\{
x\in \P^n(\Q):
H(x)\leq B
\}}
.\end{equation}

\begin{theorem} \label{thm:Delta=0}
Let $V$ be a smooth projective variety over $\QQ$ equipped with a dominant morphism $\pi: V \to \PP^n$
with geometrically integral generic fibre and $\Delta(\pi) = 0$. Then
\begin{equation} \label{def:jjlimit}
	\tau_\pi: \ZZ \to \RR, \quad j \mapsto \tau_\pi(j) := \lim_{B\to \infty} \tau_\pi(j,B)
\end{equation}
is well-defined and 
defines a probability measure on $\ZZ$. Moreover, for every $j\in \Z_{\geq 0}$ we have the following upper bound
\begin{equation}
\label{eq:wowcan}
\tau_\pi(j)
\ll_\pi \frac{1}{(1+j)^{j}
(\log (2+j))^{j/2}}
,\end{equation}
where the implied constant depends at most on $\pi$.
\end{theorem} 
One way to interpret Theorem \ref{thm:Delta=0} is that $\omega_\pi(x)$ has a \emph{limit law}.
A limit law is originally defined for functions defined in the integers (see~\cite[Def.~2.2, p.~427]{MR3363366}), 
however, the definition easily extends to functions defined in $\P^n(\Q)$:
We say that a function $f:\P^n(\Q)\to \R$ has a limit law with distribution function
$F$ if 
\[\lim_{B\to+\infty} \nu_B(x\in \P^n(\Q): f(x)\leq z) = F(z)\]
holds  for a function $F:\R\to [0,1]$ which is non-decreasing,
right-continuous and satisfies $F(-\infty)=0$, $F(+\infty)=1$,
for all $z\in \R$ for which $F$ is continuous at $z$.
The function $\omega_\pi$ takes values in $\Z_{\geq 0}$ and for such functions 
  $f:\P^n(\Q)\to \Z_{\geq 0}$ the definition 
of the limit law is equivalent to 
the existence of the limit 
\[\lim_{B\to+\infty} \nu_B(x\in \P^n(\Q): f(x)=j) \]
for every fixed $j\in \Z_{\geq 0}$ and the property \[ \sum_{j=0}^\infty
\lim_{B\to+\infty} \nu_B(x\in \P^n(\Q): f(x)=j)= 1
.\] These are the two properties that are verified in Theorem~\ref{thm:Delta=0}, in addition to a bound 
in terms of $j$
for the limits.

We illustrate Theorem \ref{thm:Delta=0} with some examples.

\begin{example} \label{ex:Ax-Kochen}
	Let $d,n > 1$. Let $\pi: V \to \PP_\Q^{N-1}$ be the family of all hypersurfaces
	of degree $d$ in $\PP_\Q^n$, where $N = {\binom{n+d}{d}}$. (Note that $V$ is regular.)
	If $(d,n) = (2,2)$, i.e.~the family of all plane conics, 
	then $\Delta(\pi) = 1/2$ \cite[Ex.~4]{Ser90} and Theorem \ref{thm:gaussian}
	applies. If however $(d,n) \neq (2,2)$, then the fibre over every codimension $1$ point is
	geometrically integral, thus $\Delta(\pi) = 0$ and Theorem \ref{thm:Delta=0} applies. 
	(See the proof of \cite[Thm.~3.6]{PV04} for this fact.)
	We deduce that when $(d,n) \neq (2,2)$,
	the probability that a smooth hypersurface has no $p$-adic point for exactly $j$ many primes $p$
	is well-defined and exists.
	
	A particularly interesting case is when $n \geq d^2$. Here 
	the Ax--Kochen theorem \cite{AK65} implies that the map $V(\QQ_p) \to \PP^{N-1}(\QQ_p)$
	is surjective
	for all but finitely many primes $p$. In particular,  we have $\tau_\pi(j)=0$
	for all but finitely many $j \in \ZZ$.
\end{example}

An example where the measure $\tau_\pi$ has infinite support is the following.

\begin{example} \label{ex:cubics}
Let 
\[
V: \quad
\sum_{i=0}^3
y_i
x_i^3=0 
\qquad
\subset
\PP^3 \times \PP^3
\]
and let $\pi:V\to \P^3$ be the projection onto the $y$-coordinate;
here $\pi$ is the family of all diagonal cubic surfaces. In \S \ref{sec:cubics} we will show 
that there exists an absolute constant $c>0$ such that 
$\tau_\pi(j) >  c (1+j)^{-3j}$ for all $j \in \ZZ_{\geq 0}$.
This 
shows that $\tau_\pi$ has
infinite support in this case and that~\eqref{eq:wowcan} cannot be significantly improved.
\end{example}

It turns out that one has the following characterisation for when the measure $\tau_\pi$ has finite support;
it happens if and only if an Ax--Kochen-type property holds.

\begin{theorem} \label{thm:Ax-Kochen}
Keep the assumptions of Theorem \ref{thm:Delta=0}. Then the measure $\tau_\pi$ has finite support 
if and only if $V(\QQ_p) \to \PP^n(\QQ_p)$ is surjective for all but finitely many primes $p$.
\end{theorem}

Families for which $V(\QQ_p) \to \PP^n(\QQ_p)$ is surjective for all but finitely many $p$
were studied in \cite{LSS17}. A geometric criterion for when this holds can be found in \cite[Thm.~1.4]{LSS17}.

Our methods also allow us to prove the following local-global principle for existence of varieties in the family which are non-locally soluble at \emph{exactly} a given finite set of places.
\begin{theorem} \label{thm:Hasse}
	Keep the assumptions of Theorem \ref{thm:Delta=0}. Let $S$ be a finite set of places of $\QQ$.
	Assume that $\pi(V(\QQ_v)) \neq \PP^n(\QQ_v)$ for all 
	$v \in S$ and that $\pi(V(\QQ_v)) \neq \emptyset$ for all
	$v \notin S$. Then there exists $x \in \P^n(\Q)$ such that
	$\pi^{-1}(x)$ is smooth and
	$$	\pi^{-1}(x)(\Q_v)= \emptyset 
	\quad \iff \quad  v \in S.
	$$
\end{theorem}

Note that for conics Hilbert's version of quadratic reciprocity implies
that a conic over $\QQ$ fails to have a $\QQ_v$-point at exactly a set of places $S$ of even cardinality,
despite there being a conic $C_v$ over every $\Q_v$ with $C_v(\Q_v) = \emptyset$.
Theorem \ref{thm:Hasse} shows that for families of varieties with $\Delta(\pi) = 0$ there is no such reciprocity law. (This phenomenon was first  observed in the case of curves of genus at least $1$ by Poonen and Stoll \cite{PS99}.)

One of the major differences 
between the case $\Delta(\pi)>0$
and $\Delta(\pi)=0$
is that the function 
$\omega_\pi(x)$
becomes
arbitrarily large on average
only
when $\Delta(\pi)>0$.
To make this precise we study the moments of $\omega_\pi$.
Define for $r\in \mathbb{Z}_{\geq 0}$
the function 
\begin{equation} \label{def:N_moment}
\c{N}_r(\pi, B)
:= 
\sum_{\substack{ x \in \PP^n(\QQ), H(x)\leq B\\ \pi^{-1}(x) \text{ smooth}}} \omega_\pi(x)^{r}.
\end{equation}
Note that an obvious
consequence of Theorem~\ref{thm:moment}
is that if $\Delta(\pi)>0$ then 
\[
\c{N}_r(\pi, B) B^{-n-1} 
\gg_{\pi,r}
(\log \log B)^r.
\]
In contrast, if $\Delta(\pi)=0$
then for all $r\geq 0$ the function
$\c{N}_r(\pi, B) B^{-n-1}$
remains bounded as $B\to\infty$;
specifically we have the following counterpart of 
Theorem~\ref{thm:moment}.
\begin{theorem} \label{thm:Deltazeromoments}
Keep the assumptions of Theorem \ref{thm:Delta=0}.
Then for every $r \in \Z_{\geq 1}$ we have 
\begin{equation}
\label{eq:partfg7}
\lim_{B\to\infty}
\frac
{\c{N}_r(\pi, B)}
{\#\{x \in \PP^n(\QQ), H(x)\leq B \}}
=\sum_{j=0}^\infty
j^r \tau_\pi(j)
.\end{equation}
\end{theorem} 
Note
that,
apart from very special cases,
existence of 
moments does not automatically imply 
existence 
of a limit law or vice versa.

\subsection{Layout of the paper and proof ingredients}
We begin in \S \ref{sec:proj} with an elementary result on counting rational points in $\PP^n(\QQ)$
which lie in a given residue class.

We prove Theorem~\ref{thm:moment} in \S\ref{sec:nerdosk}.
For this we show 
in Proposition~\ref{prop:applicationofgranvsound} 
that the moments of
a `truncated' version of $\omega_\pi$
are approximated by 
the moments of the standard normal distribution.
The proof is based on equidistribution properties of the fibers of $\pi$
that are verified during the earlier stages in \S\ref{sec:nerdosk}
and subsequently fed into work of Granville and Soundararajan \cite{MR2290492}. 
We finish the proof of 
Theorem~\ref{thm:moment}
in \S\ref{s:bachbwv997}
by showing that the moments of 
$\omega_\pi$
and the moments of 
the truncated 
version of $\omega_\pi$
have the same asymptotic behaviour. 
Theorem~\ref{thm:gaussian} 
is then deduced from Theorem~\ref{thm:moment} in~\S\ref{s:bachbwv1080}.

In \S \ref{sec:ps} we prove the results from \S\ref{sec:ps_intro}.
The most difficult part of the proof of Theorem \ref{thm:Delta=0}
is establishing the existence of the limit \eqref{def:jjlimit},
which we achieve via 
Bhargava's effective version of the Ekedahl sieve \cite{Bha14}.
Theorems \ref{thm:Ax-Kochen} and \ref{thm:Deltazeromoments} are proved using similar
methods and the results from \S \ref{sec:proj}. We finish \S \ref{sec:ps} by briefly explaining
how our results generalise in a straightforward manner to minor variants given by
considering real solubility or by dropping conditions at finitely many primes.

\begin{notation}
For an integral homogeneous polynomial $f$, a point $x \in \PP^n(\QQ)$ and $Q \in \NN$,
we say that ``$f(x) \equiv 0 \bmod Q$''	if $f(\x) \equiv 0 \bmod Q$ for some primitive representative 
$\x \in \ZZ^{n+1}$ of $x$. We use the notation ``$Q \mid f(x)$'' analogously.	

The quantities $\delta_x(\pi)$ 
and
$\Delta(\pi)$ are introduced in
Definition~\ref{def:Delta},
the function $\omega_\pi(x)$ 
is defined in~\eqref{def:omega} and
the indicator function $\theta_p(x)$ is introduced in \eqref{def:theta_p}.
The arithmetic functions
$\omega,\mu,\varphi$ 
respectively
denote 
the
number of prime divisors,
the M\"{o}bius function
and the Euler totient function, respectively.
\end{notation}

\subsection*{Acknowledgements}
We thank the referee for a careful reading of the paper and numerous suggested improvements.

\section{Explicit equidistribution on projective space} \label{sec:proj}
\subsection{Counting with congruences}
We will be required to count rational points in projective space which satisfy imposed
congruence conditions. To state our result, we let
\begin{equation} \label{def:cnn}	
c_n = \lim_{B \to \infty} 
\frac{\#\{ x \in \P^n(\Q): H(x) \leq B\} }{B^{n+1}}
 = \frac{2^{n}}{\zeta(n+1)}
,\end{equation}
where $\zeta(s)$ denotes the Riemann zeta function.
\begin{proposition} \label{prop:sieve}
	Let $B > 1$, 
	$Q \in \NN$ and 
	$\Upsilon \subseteq \PP^n(\ZZ/Q\ZZ)$. Then
	$$
	\#\left\{ x \in \PP^n(\QQ): 
	\begin{array}{l}
		H(x) \leq B, \\
		x \bmod Q \in \Upsilon
	\end{array}
	\right\} 
	= \frac{c_n \#\Upsilon }{\#\PP^n(\ZZ/Q\ZZ)} 
	B^{n+1} 
	+ O\left(Q
	\#\Upsilon
	\left( B+\frac{B^n}{Q^{n}} (\log B)^{[1/n]}
	\right) \right)
	,$$
	where $[\cdot]$ denotes the integer part.
\end{proposition}
\begin{proof} 
	Let
	$\widehat{\Upsilon } = \{\x \in (\ZZ/Q\ZZ)^{n+1} : \x \not \equiv \mathbf{0} \bmod p \, \forall p \mid Q, (x_0:\cdots:x_n) \in \Upsilon\}$
	be the affine cone of $\Upsilon$. 
	Applying M\"{o}bius inversion we see that
	the cardinality in question is
	\begin{align*}
		& \frac{1}{2}	\#\{ \x \in \ZZ^{n+1}: \max\{|x_0|, \ldots, |x_n| \} \leq B, \gcd(x_0,\ldots, x_n) = 1, \x \bmod Q \in \widehat{\Upsilon}\} + O(1)\\
		= & \frac{1}{2}	
		\sum_{\substack{k \leq B \\ \gcd(k,Q)=1  }} 
		\mu(k) \#\{ \x \in \ZZ^{n+1}: \max\{|x_0|, \ldots, |x_n| \} \leq B/k, k\x \bmod Q \in \widehat{\Upsilon}\}  + O(1) \\
		= & \frac{1}{2}	
		\sum_{\y \in \widehat{\Upsilon}} 
		\sum_{\substack{k \leq B \\  \gcd(k,Q)=1  
		}} \mu(k) 
		\#\{ \x \in \ZZ^{n+1}: \max\{|x_0|, \ldots, |x_n| \} \leq B/k, \x \equiv k^{-1}\y \bmod Q\} 
		+ O(1)  \\
		= & \frac{1}{2}	
		\sum_{\y \in \widehat{\Upsilon}} 
		\sum_{\substack{k \leq B \\  \gcd(k,Q)=1
		}} \mu(k) \left(\left( \frac{2B}{kQ}\right) + O(1) \right)^{n+1} + O(1)
	.\end{align*}	
	Continuing, we use the estimate $\sum_{k>B}k^{-n-1}\ll B^{-n}$ to find that the main term is 
	\begin{align*}
		 & \frac{2^n \#\widehat{\Upsilon}B^{n+1}}{Q^{n+1}}	\sum_{\substack{k \in \N \\ 
		\gcd(k,Q)=1
		}} 
		\frac{\mu(k)}{k^{n+1}} + O\left(\frac{\#\widehat{\Upsilon}B}{Q^{n+1}}\right)
		=  \frac{c_n \#\widehat{\Upsilon}}{Q^{n+1}\prod_{p \mid Q}(1 - 1/p^{n+1})} \cdot B^{n+1}
		 + O\left(\frac{\#\widehat{\Upsilon}B}{Q^{n+1}}\right).
	\end{align*}
	Recall that \begin{equation}\label{eq:chronic_infection_pestilence}
	\#\PP^n(\ZZ/p^k\ZZ) = p^{n(k-1)} \cdot \frac{p^{n+1} -1}{p-1} .\end{equation}
	From this it is easy to establish
	\[
	Q^{n+1}\prod_{p \mid Q}(1 - 1/p^{n+1}) = \varphi(Q) \#\PP^n(\ZZ/Q\ZZ)
	,\] 
	where $\varphi$ is Euler's totient function.
	We obtain the main term as stated in our proposition via
	$\#\widehat{\Upsilon} = \varphi(Q) \#\Upsilon$.
       To deal with the error terms observe that 
\[
 \left(\left( \frac{2B}{kQ}\right) + O(1) \right)^{n+1}
-
\left( \frac{2B}{kQ}\right)^{n+1}
\ll 
\sum_{0\leq \ell \leq n}
\left( \frac{B}{kQ}\right)^{\ell} 
\ll 1+ \left( \frac{B}{kQ}\right)^{n} 
,\] 
where the last inequality stems from
$\sum_{0\leq \ell \leq n} t^\ell \leq (n+1)(1+t^n)$,
valid for every $t\geq 0$ and $n\in \N$.
We obtain the error term
\[  \sum_{\y \in \widehat{\Upsilon}} 
      \sum_{\substack{k \leq B \\  \gcd(k,Q)=1}} |\mu(k)|
     \Bigg\{ \left( \left( \frac{2B}{kQ}\right) + O(1) \right)^{n+1} -\left( \frac{2B}{kQ}\right)^{n+1}\Bigg\}
\ll
\#\widehat{\Upsilon}
 \Big\{B+(B/Q)^n (\log B)^{[1/n]} \Big\}
.\]
Using $\#\widehat{\Upsilon} \ll Q \#\Upsilon$ completes the proof.
\end{proof}

Before continuing we record an elementary lemma here.

\begin{lemma}
\label{lem:techn}
For 
$n,Q\in \N$
we have
\[
Q^n
\leq
\#\PP^n(\ZZ/Q\ZZ)
\leq
2^{\omega(Q)} Q^n
.\] 
\end{lemma}
\begin{proof}
It suffices to prove the result when $Q = p^k$ for some prime $p$.
By~\eqref{eq:chronic_infection_pestilence}
we have \[ p^{kn} \leq \#\PP^n(\ZZ/p^k\ZZ) = p^{n(k-1)} \cdot \frac{p^{n+1} -1}{p-1} \leq p^{n(k-1)} \cdot (2p^n) = 2p^{kn}. \qedhere\]
\end{proof}

\subsection{Some probability measures} \label{sec:mu_p}
\subsubsection{Measures on $\PP^n(\Q_p)$}
Let $p$ be a prime. The finite sets $\PP^n(\ZZ/p^k\ZZ)$ come with a natural uniform probability measure.
Taking the limit $\PP^n(\ZZ_p) = \lim_{k \to \infty} \PP^n(\ZZ/p^k\ZZ)$ we obtain a well-defined probability
measure $\vartheta_p$ on $\PP^n(\ZZ_p)$ (this measure differs from Peyre's local Tamagawa measure \cite[\S2.2]{peyre} by a constant).
These measures admit the following explicit description. Let $\Upsilon \subset \PP^n(\ZZ/p^k\ZZ)$. Then
\begin{equation} \label{def:special}
	\vartheta_p\left(\{ x \in \PP^n(\QQ_p) : x \bmod p^k \in \Upsilon\} \right) = \frac{\#\Upsilon}{\#\PP^n(\ZZ/p^k\ZZ)}.
\end{equation}
These ``residue disks'' generate the $\sigma$-algebra on $\PP^n(\ZZ_p)$, hence the measure $\vartheta_p$ is uniquely determined by \eqref{def:special}.
Proposition \ref{prop:sieve} may  be viewed as an effective version of equidistribution
of rational points on $\PP^n$ with respect to the measures $\vartheta_p$.

One relates the measure $\vartheta_p$ to the usual Haar measure on $\ZZ_p$ via the following.

\begin{lemma} \label{lem:Haar}
	Let $\mu_p$ denote the Haar probability measure on $\ZZ_p^{n+1}$.
	Let $\Upsilon \subset \PP^n(\ZZ_p)$ and let $\widehat{\Upsilon} \subset \QQ_p^{n+1}$
	be its affine cone. Then $\vartheta_p(\Upsilon) = \mu_p(\widehat{\Upsilon} \cap \ZZ_p^{n+1})$.
\end{lemma}
\begin{proof}
	It suffices to prove the result for the residue disks
	$$\Upsilon = \{ x \in \PP^n(\ZZ_p) : x \equiv (a_0:a_1:\dots:a_n) \bmod p^k\}$$
	for some $a_i \in \ZZ_p$; we have $\vartheta_p(\Upsilon) = 1/\#\PP^n(\ZZ/p^k\ZZ)$. Note that one of the $a_i$ may be taken to be a unit;
	for simplicity we assume this is $a_0$. A moment's thought reveals that
	$$\widehat{\Upsilon} \cap \ZZ_p^{n+1} = \left\{ \mathbf{x} \in \ZZ_p^{n+1} : \left| \frac{x_i}{p^{v_p(x_0)}} - \frac{a_ix_0}{a_0p^{v_p(x_0)}} \right|_p \leq p^{-k}, i = 1,\dots,n\right\}.$$
	Thus, by~\eqref{eq:chronic_infection_pestilence} we find that 
	\begin{align*}
		\mu_p(\widehat{\Upsilon}\cap \ZZ_p^{n+1}) & = \sum_{m = 0}^\infty 
		\{ \mathbf{x} \in \ZZ_p^{n+1} : v_p(x_0) = m, |x_i - a_i x_0/a_0|_p \leq p^{-k-m}, i = 1,\dots,n\}	\\
		& = \sum_{m = 0}^\infty \frac{1}{p^m} \left(1 - \frac{1}{p}\right) \left(\frac{1}{p^{k+m}}\right)^n	 
		 = \frac{1}{p^{n(k-1)}}\left(\frac{p-1}{p^{n+1}-1}\right) 
		= \frac{1}{\#\PP^n(\ZZ/p^k\ZZ)}. \qedhere
	\end{align*}
\end{proof}

\subsubsection{Measure on $\PP^n(\R)$}
We let $\vartheta_\infty$ be the pushforward of the usual
probability measure on the $n$-sphere $S^n$ via the quotient map $S^n \to \PP^n(\R)$.

\section{An Erd\H{o}s--Kac theorem for fibrations}
\label{sec:nerdosk}

\subsection{Set-up}
\label{s:equidistribution}
We begin the proof of the results from \S\ref{sec:CLT}.
Let $V$ be a smooth projective variety over $\QQ$  with a dominant morphism $\pi: V \to \PP^n_\QQ$
with geometrically integral generic fibre. (We assume $\Delta(\pi) > 0$ from \S\ref{s:approxim}.)
We choose a \emph{model} for $\pi$,
i.e.~a proper scheme $\mathcal{V}$ over $\mathbb{Z}$
together with a proper morphism $\pi:\mathcal{V} \to \PP^n_\ZZ$ (also denoted $\pi$ by abuse of notation),
which extends  $V \to \PP_\QQ^n$. In what follows, all implied constants are allowed
to depend on $\pi$, the choice of model, and the  $A$ and $f$
occurring in Lemma \ref{lem:f}.

We begin by studying the basic properties of $\omega_\pi(x)$. We first show that it enjoys
analogous bounds to the usual $\omega$.

\begin{lemma} \label{lem:bounds}
	There exists $D > 0$ with the following property. Let $x \in \PP^n(\QQ)$ be such that
	$\pi^{-1}(x)$ is smooth. Then
	$$\omega_\pi(x) \ll \frac{ \log H(x) }{\log \log H(x) }, \quad 
	\max \{ p : \pi^{-1}(x)(\QQ_p) = \emptyset\} \ll H(x)^D.$$
\end{lemma}
\begin{proof}
	Let $S \subset \PP^n_\QQ$ denote the non-smooth locus of $\pi$;
	this is a proper closed subset of $\PP^n_\QQ$.
	Let $\mathcal{S}$ be the closure of $S$ in $\PP^n_{\ZZ}$ and 
	choose a finite collection of homogeneous polynomials $f_1,\ldots,f_s$ 
	which generate the ideal of $\mathcal{S}$. Let $D = \max_i \{ \deg f_i \}$.
	
	For all sufficiently large primes $p$, the fibre $\pi^{-1}(x \bmod p)$ 
	is smooth if and only if $x \bmod p \notin \mathcal{S}$,
	which  happens if and only if $p \nmid f_i(x)$ for some $i$.
	Moreover, by the Lang--Weil estimates \cite{LW} and Hensel's lemma, for all sufficiently
	large primes $p$ (independently of $x$)
	if $\pi^{-1}(x \bmod p)$ is smooth then $\pi^{-1}(x)(\QQ_p) \neq \emptyset$.
	It follows that
	$$\#\{ p : \pi^{-1}(x)(\QQ_p) = \emptyset \} \leq \#\{ p : p \mid f_i(x) \,\forall \, i \in \{1,\dots,s\} \} + O(1).$$
	Letting $\x$ be a primitive representative of $x$
	and using the bound $\omega(n)\ll (\log n)(\log \log n)^{-1}$,
	we obtain
	\[\omega_\pi(x) \leq \sum_{i=1}^r \omega(f_i(\x)) + O(1) \ll \frac{ \log (H(x)^D) }{\log \log (H(x)^D)},
	\quad \max \{ p : \pi^{-1}(x)(\QQ_p) = \emptyset\} \ll H(x)^D. \qedhere\]
\end{proof}

To simplify notation it will be easier to work with some choice of polynomial which
vanishes on the singular locus, rather than the whole singular locus. The proof
of the following is a minor adaptation of the proof of Lemma \ref{lem:bounds}
(just choose $f = f_1$).

\begin{lemma} \label{lem:f}
	Let $f \in \ZZ[x_0,\ldots, x_n]$ be a homogeneous square-free polynomial such that 
	$\pi$ is smooth away from the divisor $f(x) = 0 \subset \PP^n_\QQ$.
	Then there exists $A > 0$ such that for all primes $p > A$ the following hold.
	\begin{enumerate}
		\item The restriction of $\pi$ to $\PP^n_{\FF_p}$ is smooth away from the divisor 
		$f(x) = 0 \subset \PP^n_{\FF_p}$.
		\item Let $x \in \PP^n(\Q)$. If $\pi^{-1}(x)(\QQ_p) = \emptyset$ then $p \mid f(x)$.
	\end{enumerate}
\end{lemma}

In \S\ref{sec:3.2}, we allow ourselves to increase $A$ as necessary to take care of bad behaviour at small primes. 

\subsection{Equidistribution properties in the fibres} \label{sec:3.2}
The next step is to translate the condition 
$\pi^{-1}(x)(\Q_p)=\emptyset$ into something amenable to tools from analytic number theory.
We do this by using the tools developed in \cite{LS16}.
The key result is \cite[Thm.~2.8]{LS16}, which is a valuative criterion
for non-existence of a $p$-adic point in a fibre, for sufficiently large primes $p$.
In the special case of the conic bundle over $\QQ$
$$x^2 + y^2 = tz^2 \qquad \subset \PP^2 \times \mathbb{A}^1,$$
the criterion \cite[Thm.~2.8]{LS16} says that if $p \equiv 3 \bmod 4$ and the $p$-adic
valuation of $t$ is $1$, then the fibre over $t$ has no $\QQ_p$-point, as is familiar from the theory of Hilbert symbols.

We introduce the quantity which will arise in this analysis. For any prime $p$ let
\begin{equation} \label{def:sigma_p}
\sigma_p:=
\frac{\#\big\{x \in \P^n(\F_{p}): \pi^{-1}(x) \mbox{ is non-split}\big\}}{\#\PP^n(\FF_p)}.
\end{equation}
(Here we use the term ``non-split'' in the sense of Skorobogatov \cite[Def.~0.1]{Sko96}.)

\begin{lemma} \label{lem:harmonic}
	Let $A$ and $f$ be as in Lemma \ref{lem:f} and $p > A$. Then
	\[
	\sigma_p \leq  \frac{d}{p}, \qquad \text{where } d = \deg f.
	\]
\end{lemma}
\begin{proof}
	A  non-split fibre is necessarily singular. Hence by Lemma \ref{lem:f}, for $p > A$ we have
	$$\#\big\{x \in \P^n(\F_{p}): \pi^{-1}(x) \mbox{ is non-split}\big\} 
	\leq \#\big\{x \in \P^n(\F_{p}): f(x) = 0\big\}.$$
	Projecting to a suitable hyperplane, this  is at most
	$d \#\P^{n-1}(\F_{p})$. The result follows.
\end{proof}

We now use results from \cite{LS16} to deduce an equidistribution result for $\theta_p$.
To simplify notation, we denote the characteristic function of the $p$-adically insoluble fibres by
\begin{equation} \label{def:theta_p}
\theta_p(x):=
\begin{cases} 1, &\mbox{if } \pi^{-1}(x)(\Q_p)=\emptyset,\\ 
0, & \mbox{otherwise}.\end{cases}
\end{equation}  

Our result is the following asymptotic upper and lower bounds. 
(Here $c_n$ is as in \eqref{def:cnn}.)
\begin{proposition}
\label{prop:arcturus} 
Let $d = \deg f$.
Enlarging $A$ if necessary, 
there exists $\alpha \geq 0$ 
with the following property.
Let $Q \in \NN$
be square-free with $p \nmid Q$ for all $p \leq A$. 
Then 
\[\pm \hspace{-15pt}\sum_{\substack{ x \in \PP^n(\QQ) \\ H(x)\leq B\\ \pi^{-1}(x) \text{ smooth}}}
\hspace{-10pt}
\prod_{p \mid Q}\theta_{p}(x)
\leq \pm c_n
\Big(\prod_{p \mid Q}
(\sigma_{p} \pm \alpha/p^2)
\Big)
B^{n+1} + O\left(
(4d)^{\omega(Q)}  
	( Q^{2n+1} B + Q B^n(\log B)^{[1/n]} )
	\right),
	\]
where the implied constant is independent of $B$ and $Q$.
\end{proposition}
\begin{proof}
	Let $p > A$ be a prime.
	Enlarging $A$ if necessary, the Lang--Weil estimates and Hensel's lemma show that
	if 	$\pi^{-1}(x \bmod p)$ is split then  $\pi^{-1}(x)$ has a
	$\QQ_p$-point. 
	Thus the sum in the proposition
	 is
	\[
	\leq 
	 \#\{ x \in \PP^n(\QQ):H(x) \leq B, \pi^{-1}(x \bmod p) \mbox{ is non-split} \, \forall p\mid Q\}
	.\]
	Applying Proposition~\ref{prop:sieve} 
	with 
	$\Upsilon=\{x \in \PP^n(\ZZ/Q\ZZ):\pi^{-1}(x \bmod p) \mbox{ is non-split} \, \forall p\mid Q\}$,
	we infer that the above cardinality equals
	\[c_n B^{n+1}\prod_{p\mid Q}
	\frac{\#\big\{x \in \P^n(\F_{p}): \pi^{-1}(x) \mbox{ is non-split}\big\}}{\#\PP^n(\F_{p})} 
		+O\left(
		Q
	\#\Upsilon
	\left( B+\frac{B^n}{Q^n} (\log B)^{[1/n]}
	\right)
		\right).
	\]
	One has $\#\Upsilon=\#\PP^n(\ZZ/Q\ZZ) \prod_{p\mid Q} \sigma_p$, thus
	Lemmas~\ref{lem:techn} and~\ref{lem:harmonic} imply that
	$\#\Upsilon
	\ll (2d)^{\omega(Q)}	Q^{n-1}$, 
	which is satisfactory for the upper bound in the proposition.
		
	For the lower bound, we apply the sparsity result of \cite[Thm.~2.8]{LS16}.
	This gives a square-free homogeneous polynomial $g \in \ZZ[x_0,\ldots,x_n]$ 
	which is coprime with $f$ and contains the singular locus of $f$ such that,
	enlarging $A$ if necessary, the sum in the proposition is
	$$ 
		\geq  
	\#\left\{x \in \PP^n(\QQ):
		\begin{array}{l}
			H(x) \leq B, f(x) \equiv 0 \bmod p, f(x) \not \equiv 0 \bmod p^2,   \\
			g(x) \not \equiv 0 \bmod p, \pi^{-1}(x \bmod p) \mbox{ is non-split}, \forall p\mid Q 	
		\end{array}	\right\}
	.$$ 
	We now apply Proposition \ref{prop:sieve} with
	\[
	\Upsilon'=\left\{x \in \P^n(\Z/{Q}^2\Z):
		\begin{array}{l}
			f(x) \equiv 0 \bmod p, f(x) \not \equiv 0 \bmod p^2,  \\
			g(x) \not \equiv 0 \bmod p, \pi^{-1}(x \bmod p) \mbox{ is non-split}, \, \forall p \mid Q	
		\end{array}	\right\}
	\]
	to see that the sum in the proposition is 
	\begin{align} \label{eqn:lower_bound}
		\geq  \frac{c_n \#\Upsilon'}{\#\PP^n(\Z/Q^2\Z) }B^{n+1}
		+O \left(Q^2 \#\Upsilon'
	\left( B+\frac{B^n}{Q^{2n}} (\log B)^{[1/n]}
	\right)\right).
	\end{align}
	As $g(x) = 0$ contains the singular locus of $f$, we may apply \cite[Prop.~2.3]{BL17}
	to find that
	\begin{align*}
		& \#\left\{x \in \P^n(\Z/{p^2}\Z):
		\begin{array}{l}
			f(x) \equiv 0 \bmod p, f(x) \not \equiv 0 \bmod p^2,  \\
			g(x) \not \equiv 0 \bmod p, \pi^{-1}(x \bmod p) \mbox{ is non-split}		
		\end{array}	\right\} \\
		& = \#\left\{x \in \P^n(\FF_{p}):
				f(x) = 0, g(x) \neq  0, \pi^{-1}(x) \mbox{ is non-split}		
			\right\} (p^n + O(p^{n-1})) \\
		& = \#\{x \in \P^n(\F_{p}): \pi^{-1}(x) \mbox{ is non-split}\}p^{n} + O(p^{2n-2}).
	\end{align*}
	Here the last line follows from the fact that if $\pi^{-1}(x)$  is non-split
	then necessarily $f(x) = 0$ by Lemma \ref{lem:f}, together with the fact that
	$\#\{ x \in \PP^n(\FF_p): f(x) = g(x) = 0 \} \ll p^{n-2}$ as $f$ and $g$ share
	no common factor. Recalling that $\#\PP^n(\Z/p^2\Z) = p^n \#\PP^n(\FF_p)$
	\eqref{eq:chronic_infection_pestilence}, 
	shows that 
	\begin{equation} \label{eqn:alpha}
	\prod_{p\mid Q} (\sigma_p - \alpha/p^{2}) \leq \frac{\#\Upsilon'}{\#\PP^n(\Z/Q^2\Z) } \leq 
	\prod_{p\mid Q} (\sigma_p + \alpha/p^{2})
	\end{equation}
	for some $\alpha > 0$. This yields the correct main term for the lower bound.
	For the error term, enlarging $A$ if necessary we have $\alpha/p^2 < d/p$ for all $p > A$.
	Thus \eqref{eqn:alpha} and Lemmas~\ref{lem:techn} and \ref{lem:harmonic} give
	$$\#\Upsilon' \leq \#\PP^n(\Z/Q^2\Z)  \prod_{p\mid Q} (\sigma_p + \alpha/p^{2})
	\leq (2^{\omega(Q)}Q^{2n}) \cdot (2d)^{\omega(Q)}/Q,$$
	which yields the required error term in \eqref{eqn:lower_bound}.
\end{proof}

We now fix the choice of $A$.
Note that as $\min \{p : p \mid Q\} \to \infty$,
the lower bound in Proposition \ref{prop:arcturus}
converges to the upper bound.

\begin{lemma}
\label{lem:pnt}
We have
$$ \sum_{p \leq B} \#\{ x \in \PP^n(\FF_p) : \pi^{-1}(x) \mbox{ is non-split}\} 
= 
\Delta(\pi) \frac{B^{n}}{\log(B^n)} + O\l(\frac{B^{n}}{(\log B)^2}\r).$$ 

\end{lemma}
\begin{proof}
	This follows from an easy modification of the proof of \cite[Prop.~3.10]{LS16}.
	(\emph{Loc.~cit.} states an asymptotic formula without an  error term; one
	obtains an error term via the version of the Chebotarev density theorem
	given in~\cite[Thm.~9.11]{Ser12}.)
\end{proof}

We will also require the following.

\begin{proposition} \label{prop:abel}
There exists a constant $\beta_\pi$ such that
\[
\sum_{p\leq B}
\sigma_p
=\Delta(\pi) (\log \log B)+
\beta_\pi
+O((\log B)^{-1})
\]
\end{proposition}
\begin{proof}
The proof is a simple application of Lemma \ref{lem:pnt} and partial summation.
First, 
let  $a_p:=\#\{ x \in \PP^n(\FF_p) : \pi^{-1}(x) \mbox{ is non-split}\}$,
define $S(B):=\sum_{2<p\leq x}a_p$ and let 
\[R(B):=S(B) 
-
\Delta(\pi) \frac{B^{n}}{\log(B^n)}. 
\]  
Lemma~\ref{lem:pnt}  is equivalent to the estimate 
$R(B)\ll   \frac{B^{n}}{(\log B)^2}.$ By partial summation we obtain
\[
\sum_{2<p \leq B} \frac{a_p}{p^n}
=\frac{S(B)}{B^n}+
n\int_2^t \frac{\l(\Delta(\pi) \frac{u^n}{ \log (u^n)} +R(u) \r)}{u^{n+1}} \mathrm{d}u
.\] Lemma~\ref{lem:pnt} directly gives $S(B)/B^n\ll 1/\log B$.
We furthermore have 
$
\int_2^\infty 
\frac{|R(u)|}{u^{n+1}} \mathrm{d}u
<\infty$
due to $|R(B)|\ll   \frac{B^{n}}{(\log B)^2}.$ Hence, we may write 
\begin{align*}
\int_2^t \frac{R(u)}{u^{n+1}} \mathrm{d}u 
=&\int_2^\infty \frac{R(u)}{u^{n+1}} \mathrm{d}u
+O\l(\int_t^\infty  \frac{|R(u)|}{u^{n+1}}  \mathrm{d}u\r)
\\=&
\int_2^\infty \frac{R(u)}{u^{n+1}} \mathrm{d}u
+O\l(\int_t^\infty  \frac{1}{u (\log u)^2} \mathrm{d}u\r)
\\=&
\int_2^\infty \frac{R(u)}{u^{n+1}} \mathrm{d}u
+O\l(\frac{1}{\log t}\r)
.\end{align*}
Recalling  that $\int_2^t (u  \log u)^{-1}\mathrm{d}t=\log \log t-\log \log 2$ and letting 
 $$\gamma_\pi:=
-\Delta(\pi) (\log\log 2)
+\frac{a_2}{2^n}+
\int_2^\infty \frac{R(u)}{u^{n+1}} \mathrm{d}u
$$
 we have proved 
\begin{equation}
\label{eq:fugue546}
\sum_{p\leq B} \frac{a_p}{p^n}=
\Delta(\pi)\log \log B
+\gamma_p
+O\l(\frac{1}{\log B}\r)
.\end{equation}
Using this in the simple form  $\sum_{p\leq B} \frac{a_p}{p^n} \ll \log \log B$
then by partial summation we obtain  
\[
\sum_{p\leq B} \frac{a_p}{p^{n+1}}
\ll
\frac{\log \log B}{B}
+\int_1^B\frac{\log \log u}{u^2}
\mathrm{d}u
\ll 1.\]
We thus obtain that $\sum_p \frac{a_p}{p^{n+1}}$ converges and that the tail is at most 
\begin{equation}
\label{eq:fugue5467}
  \sum_{p>B}  \frac{a_p}{p^{n+1}} \ll \frac{\log \log B}{B}+
\int_B^\infty
\frac{\log \log u}{u^2}
\mathrm{d}u
\ll \frac{\log \log B}{B}
.\end{equation}
Let us now define the function $\epsilon_p$ for primes $p$ via the equation
\[
\frac{1}{\#\PP^n(\FF_p)}
=\frac{1}{p^n}
+
\frac{\epsilon_p}{p^n}
.\] Recalling~\eqref{def:sigma_p} and making use of~\eqref{eq:fugue546}, 
we see that this gives 
\[
\sum_{p\leq B}\sigma_p=
\sum_{p\leq B}\frac{a_p}{p^n}
+
\sum_{p\leq B}\frac{a_p \epsilon_p}{p^n}
=
\Delta(\pi)\log \log B
+\gamma_p
+O\l(\frac{1}{\log B}\r)
+
\sum_{p\leq B}\frac{a_p \epsilon_p}{p^n}
.\] At this point we use~\eqref{eq:chronic_infection_pestilence} to obtain  
$$
\frac{1}{\#\PP^n(\FF_p)}
=\frac{p-1}{p^{n+1}-1}
=\frac{1}{p^n} \frac{1}{(1+\frac{1}{p}+\cdots+\frac{1}{p^n})}
=\frac{1}{p^n} \l(1+O\l(\frac{1}{p}\r)\r)
,$$ from which we get  
 $\epsilon_p\ll 1/p$.
By~\eqref{eq:fugue5467}
we  
see that $\sum_{p } a_p\epsilon_p p^{-n}$ converges and that  
\[  
\sum_{p\leq B}  \frac{a_p\epsilon_p}{p^{n}} 
=
\sum_{p }  \frac{a_p\epsilon_p}{p^{n}} 
+ O\l(
\frac{\log \log B}{B}
\r)
.\]
Taking $\beta_\pi:=\gamma_\pi+\sum_p  a_p\epsilon_p p^{-n} $
concludes the proof. 
\end{proof}


\subsection{Moments of a truncated version of $\omega_\pi$.}
\label{s:approxim} 
We assume from now on that $\Delta(\pi) > 0$.
In what follows
$t_0,t_1:\R\to \R$ 
are two functions 
that  satisfy
\[
1<t_0(B)<t_1(B)<B,
\lim_{B\to\infty} t_0(B)=
\lim_{B\to\infty} t_1(B)=
\infty
.\] 
Both functions $t_0(B),t_1(B)$
will be chosen optimally at a later stage.
Define the function  
\begin{equation} 
\label{def:truncation}  \omega_\pi^\flat(x,B) := \sum_{t_0(B)<p \leq t_1(B)} 
(\theta_p(x)-\sigma_p),\end{equation} where $\sigma_p$ is as in \eqref{def:sigma_p}. 
We  need to estimate asymptotically the moments of 
$\omega_\pi(x)$;
it turns out that 
it is easier to work with 
the `truncated' version $\omega_\pi^\flat(x,B)$ of $\omega_\pi(x)$.
Introducing $t_1(B)$
deals with the dependence on $Q$ in the error term of Proposition~\ref{prop:arcturus},
whilst $t_0(B)$
is used to control the error $\pm \alpha /p^{2}$
in the leading constant in Proposition~\ref{prop:arcturus}.

To study the degree to which 
 $\omega_\pi(x)$ is affected by the primes in the interval $(t_0(B),t_1(B)]$
we begin by observing that Proposition \ref{prop:abel}
provides us with
\begin{equation}
\label{eq:nbound} 
\sum_{p\leq t_0(B)}\sigma_p\ll 
\log \log t_0(B), \qquad
\sum_{t_1(B)<p\leq B}\sigma_p\ll 
\Big(\log \frac{\log B}{\ \log t_1(B)}\Big)+\frac{1}{\log t_1(B)}   
.\end{equation}
We define
\begin{equation} \label{def:truncated_moment}
\c{M}^\flat_r(\pi, B):=\sum_{\substack{ x \in \PP^n(\QQ), H(x)\leq B\\ \pi^{-1}(x) \text{ smooth}}} 
\l(\frac{\omega_\pi^\flat(x,B) }{\sqrt{\Delta(\pi)\log \log B}}\r)^{r}
, (r \in \Z_{\geq 0}).
\end{equation}
Note that 
$\c{M}^\flat_r(\pi, B)$
depends on $t_0(B)$ and $t_1(B)$ due to the presence of 
$\omega_\pi^\flat(x,B)$. 
The estimates in Proposition \ref{prop:abel} and \eqref{eq:nbound} yield \begin{equation} \label{eq:varcheb} 
\sum_{t_0(B)<p\leq t_1(B)} \sigma_p =\Delta(\pi) \log \log B 
+O\Big(\Big(\log \frac{\log B}{\ \log t_1(B)}\Big)  + \log \log t_0(B)\Big) .\end{equation} 
Furthermore, we have  $\sigma_p-\sigma_p^2=\sigma_p+O(p^{-2})$ due to Lemma \ref{lem:harmonic}. 
This shows that \begin{equation} \label{eq:varcheb2} \sum_{t_0(B)<p\leq t_1(B)} (\sigma_p-\sigma_p^2) =
\Delta(\pi) \log \log B +O\Big( \Big(\log \frac{\log B}{\ \log t_1(B)}\Big)  + \log \log t_0(B) \Big) .\end{equation}

\begin{lemma}
\label{lem:compareotpr}
Let $\psi \in \{ -1,1\}$.
In the situation of Proposition~\ref{prop:arcturus} we have 
\[
\prod_{p \mid Q}
(\sigma_{p} 
+\psi 
\alpha/p^2)
=
\prod_{p \mid Q}
\sigma_{p} 
+ O\left(
\frac{(2\alpha d)^{\omega(Q)}}{Q \min\{p: p\mid Q\}}
	\right)
	.\]
\end{lemma}
\begin{proof}
We prove the inequality for one choice of sign,
namely $\psi=1$, 
the other choice being similar.
Denoting
$Q=p_1\cdots p_k$ with $p_i<p_{i+1}$, 
we have by Lemma~\ref{lem:harmonic} 
that 
\begin{align*} 
& \Big(\prod_{p|Q} (\sigma_p+\alpha/p^2)\Big)-\Big(\prod_{p|Q}\sigma_p\Big)
=\sum_{S\subsetneq \{1,\ldots,k\}} \Big(\prod_{i\in S}\sigma_{p_i}\Big)\Big(\prod_{i\notin S}\frac{\alpha}{p_i^2}\Big)
\\
&\leq \sum_{S\subsetneq \{1,\ldots,k\}} 
\Big(\prod_{i\in S}\frac{d}{p_i}\Big)\Big( \prod_{i\notin S}\frac{\alpha}{p_i^2}\Big)
\leq (\alpha d)^k\sum_{S\subsetneq \{1,\ldots,k\}} 
\Big(\prod_{i\in S}\frac{1}{p_i}\Big)\Big( \prod_{i\notin S}\frac{1}{p_i^2}\Big) \\
& = \frac{(\alpha d)^k}{Q}\sum_{S\subsetneq \{1,\ldots,k\}} 
\Big( \prod_{i\notin S}\frac{1}{p_i}\Big) \leq  \frac{(2\alpha d)^k}{Qp_1} \qedhere
.\end{align*}
\end{proof} 
Before proceeding we recall~\cite[Prop.~3]{MR2290492}.

\begin{lemma}[Granville--Soundararajan]
\label{lem:generalerdoskac}
Let $\c{P}$ be a finite set of primes and let $\c{A}:=\{a_1,\ldots,a_y\}$ be a multiset of $y$ natural numbers. For $Q\in \N$ define
\[\c{A}_Q:=\#\{m\leq y:Q \mid a_m\}.\]
Let $h$ be a real-valued, non-negative multiplicative function  such that 
for square-free $Q$ we have $0\leq h(Q)\leq Q$. For any $r\in \N$ we let 
\[ C_r=\Gamma(r+1)/(2^{r/2}\Gamma(1+r/2))
\
\text{ and } 
\
\c{E}_\c{P}(\c{A},h,r):=\sum_{\substack{Q\in \N, \mu(Q)^2=1 \\ p\mid Q\Rightarrow p\in \c{P}\\ \omega(Q)\leq r }} \bigg| \c{A}_Q-\frac{h(Q)}{Q} y\bigg|
.\]
Defining 
\[\mu_\c{P}(h):=\sum_{p\in \c{P}} \frac{h(p)}{p}
\
\text{ and } 
\
\sigma_\c{P}(h):=\bigg(\sum_{p\in \c{P}} \frac{h(p)}{p}\Big(1-\frac{h(p)}{p}\Big)\bigg)^{1/2},\]
we have uniformly for all $r\leq \sigma_\c{P}(h)^{2/3}$ that 
\begin{equation}\label{eq:evenmoments}\sum_{a\in \c{A}} \big(\#\{ p\in \c{P}:p\mid a\}-\mu_\c{P}(h)\big)^r=C_r y \sigma_\c{P}(h)^r 
+O(
C_r y \sigma_\c{P}(h)^{r-2} r^3
+
\mu_\c{P}(h)^r \c{E}_\c{P}(\c{A},h,r))
\end{equation}
if $r$ is even, and 
\begin{equation}\label{eq:oddmoments}
\sum_{a\in \c{A}} \big(\#\{ p\in \c{P}:p\mid a\}-\mu_\c{P}(h)\big)^r\ll
C_r y \sigma_\c{P}(h)^{r-1} r^{3/2}+
\mu_\c{P}(h)^r \c{E}_\c{P}(\c{A},h,r)
\end{equation}
if $r$ is odd.
\end{lemma}

We apply this result to study the moments of $ \omega_\pi^\flat(x,B)$. 

\begin{proposition}
\label{prop:applicationofgranvsound}
Fix a positive integer $r$ and let 
$t_0(B)$ and $t_1(B)$ be given by
\begin{equation}
\label{dem:growth234}
t_0(B):=(\log \log B)^{2r} \ \text{ and } \  t_1(B)=B^{\frac{1}{5r(n+1)}}
.\end{equation}
Then we have 
\begin{numcases}{\frac{\c{M}^\flat_r(\pi, B)}{c_n B^{n+1}}=}
\mu_r +O_{r}
\Bigg( \frac{\log \log \log \log B}{\log \log B}\Bigg),&
 \textrm{ if $r$ is even}, \label{eq:evenmo}\\  
 O_{r}
\Bigg(\frac{1 }{   (\log \log B)^{\frac{1}{2}}}\Bigg),  & \textrm{ if $r$ is odd}. \label{eq:oddmo}
 \end{numcases} 
\end{proposition}

 \begin{proof}  
 We apply Lemma~\ref{lem:generalerdoskac} with 
 \[\c{A}:=\Big\{a_x := \,\,\prod_{\mathclap{\substack{p \text{ prime }\\\pi^{-1}(x)(\Q_p)=\emptyset}}} \,\,p:
 x \in \PP^n(\QQ), H(x)\leq B, \pi^{-1}(x) \text{ smooth}\Big\}.\] 
 Lemma~\ref{lem:bounds} ensures that $a_x$ is well-defined. The key property of $a_x$ is that 
for any square-free $Q$ we have 
\[Q\mid a_x \iff (p\mid Q \Rightarrow \pi^{-1}(x)(\Q_p)=\emptyset).\] Therefore, if we let \[\c{P}:=\{p \in (t_0,t_1]:p \text{ prime}\}, h(Q):=Q\prod_{p
\mid Q} \sigma_p,y:=\#\{x \in \PP^n(\QQ): H(x)\leq B, \pi^{-1}(x) \text{ smooth}\},\] then \[\c{A}_Q-\frac{h(Q)}{Q} y=
\Bigg(\sum_{\substack{ x \in \PP^n(\QQ) \\ H(x)\leq B\\ \pi^{-1}(x) \text{ smooth}}}\prod_{p \mid Q}\theta_{p}(x)\Bigg)
-\Big(\prod_{p\mid Q} \sigma_p\Big) y.\] 
Note that $y=c_nB^{n+1}+O(B^n (\log B)^{[1/n]})$; indeed
$$\#\{x \in \PP^n(\QQ): H(x)\leq B, \pi^{-1}(x) \text{ singular}\} \leq 
\#\{x \in \PP^n(\QQ): H(x)\leq B, f(x) = 0\} \ll B^{n}$$
by Lemma \ref{lem:f} and \cite[Thm.~13.4]{Ser97}.
To study $\c{E}_\c{P}(\c{A},h,r)$ we  use this and  Lemma~\ref{lem:harmonic} to show that if $Q$ is square-free and is divided only by primes $p>A$, then 
\[\c{A}_Q-\frac{h(Q)}{Q} y=\Bigg(\sum_{\substack{ x \in \PP^n(\QQ) \\ H(x)\leq B\\ \pi^{-1}(x) \text{ smooth}}}
\prod_{p \mid Q}\theta_{p}(x)\Bigg)-\Big(\prod_{p\mid Q} \sigma_p\Big) c_n B^{n+1}+O\Bigg(
d^{\omega(Q)}
\frac{B^n (\log B)^{[1/n]}}{Q}\Bigg).\] 
We can now employ Proposition~\ref{prop:arcturus} and Lemma~\ref{lem:compareotpr} to see that 
\[\sum_{\substack{ x \in \PP^n(\QQ) \\ H(x)\leq B\\ \pi^{-1}(x) \text{ smooth}}}\prod_{p \mid Q}\theta_{p}(x)-
\Big(\prod_{p\mid Q} \sigma_p\Big) c_n B^{n+1} \ll \frac{B^{n+1}(2\alpha d)^{\omega(Q)}}{Q \min\{p: p\mid Q\}}
+(4d)^{\omega(Q)}( Q^{2n+1} B + Q B^n(\log B)^{[1/n]} ).\]
Noting that $A^{\omega(Q)}/Q\ll Q^{-0.9}\ll Q^{0.9}\ll (4d)^{\omega(Q)}Q$ we deduce that 
\[\Bigg|\c{A}_Q-\frac{h(Q)}{Q} y\Bigg| \ll \frac{B^{n+1}(2\alpha d)^{\omega(Q)}}{Q \min\{p: p\mid Q\}}
+(4d)^{\omega(Q)}( Q^{2n+1} B + Q B^n(\log B)^{[1/n]} ).\] 
For any square-free $Q$ that is divisible by at most $r$ primes, all lying in $(t_0,t_1]$, we have $Q\leq t_1^r$. Therefore, in the notation
of Lemma~\ref{lem:generalerdoskac} we have \[\c{E}_\c{P}(\c{A},h,r)\ll_r B^{n+1} \Bigg(
\sum_{\substack{Q\in \N, \mu(Q)^2=1 \\ p\mid Q\Rightarrow p\in \c{P}\\ \omega(Q)\leq r }}  \frac{1}{Q \min\{p: p\mid Q\}}\Bigg)
+ (t_1^{r(2n+1)} B + t_1^r B^n(\log B)^{[1/n]} ) t_1^r ,\] where we used the estimate  
\[\sum_{\substack{Q\in \N, \mu(Q)^2=1 \\ \omega(Q)\leq r, p\mid Q\Rightarrow p\in \c{P} }} 1
\leq \#\{Q\in \N:Q\leq t_1^r\} =t_1^r.
\] Writing $Q=p_1\cdots p_r$ with $p_i<p_{i+1}$ we have 
\[
\sum_{\substack{Q\in \N, \mu(Q)^2=1 \\ \omega(Q)\leq r, p\mid Q\Rightarrow p\in \c{P} }}  \frac{1}{Q \min\{p: p\mid Q\}}
=\sum_{t_0<p_r\leq t_1}\frac{1}{p_r} \sum_{t_0<p_{r-1}<p_r}\frac{1}{p_{r-1}}\cdots \sum_{t_0<p_{2}<p_3}\frac{1}{p_{2}}
\sum_{t_0<p_{1}<p_2}\frac{1}{p_{1}^2}, \] which can be seen to be $\ll \frac{(\log \log t_1)^{r-1}}{t_0}$ due to $\sum_{p\leq t_1} p^{-1}\ll \log \log 
t_1$ and $\sum_{p> t_0} p^{-2}\ll t_0^{-1}$. Using  
assumption~\eqref{dem:growth234} 
we obtain \begin{align*}\c{E}_\c{P}(\c{A},h,r)
&\ll_{r}
\frac{B^{n+1} (\log \log B)^{r-1} }{t_0} + (t_1^{r(2n+1)} B + t_1^r B^n(\log B)^{[1/n]} ) t_1^r \\
&\ll_{r}
B^{n+1} (\log \log B)^{-r-1}
+B^{n+1/2} \\
&\ll_{r}
B^{n+1} (\log \log B)^
{-r-1}
.\end{align*}
Define  \[ 
\hat{\mu}(B):=\sum_{t_0<p\leq t_1} \sigma_p \ \text{ and } \ \hat{\sigma}(B):=\Big(\sum_{t_0<p\leq t_1} 
(\sigma_p-\sigma_p^2)\Big)^{1/2}.\] Note that by~\eqref{eq:varcheb}-\eqref{eq:varcheb2} and~\eqref{dem:growth234}
we have
 \begin{equation}\label{eq:shouldsp}
\hat{\mu}(B)=\Delta(\pi) \log \log B+O
(\log \log \log \log B).\end{equation}
Furthermore, Lemma~\ref{lem:harmonic} 
shows that $\sum_p \sigma_p^2=O(1)$, hence \[\hat{\sigma}(B)^2=\hat{\mu}(B)+O_{r}
(1).\]
By~\eqref{eq:shouldsp} 
we get 
$\hat{\sigma}(B)=(\Delta(\pi) \log \log B)^{1/2} 
(1+O_r((\log \log \log \log B)/\log \log B))^{1/2}
$, hence using  the estimate $(1+\epsilon)^{1/2}=1+O(\epsilon)$
that is 
valid for all $0<\epsilon <1$, 
we obtain
 \begin{equation}\label{eq:shouldspfgy67}
\hat{\sigma}(B)
=(\Delta(\pi) \log \log B)^{1/2}+O_r\l(\frac{\log \log \log \log B}{(\log \log B)^{1/2}}\r)
.\end{equation}
We therefore see that 
the error term in~\eqref{eq:evenmoments} is 
\[\ll_{r}
B^{n+1} (\log \log B)^{r/2-1} + B^{n+1} (\log \log B)^{-1}
\ll_{r}
B^{n+1} (\log \log B)^{r/2-1} 
.\]
Noting that \[\sum_{a\in \c{A}} \big(\#\{ p\in \c{P}:p\mid a\}-\mu_\c{P}(h)\big)^r=\sum_{\substack{ x \in \PP^n(\QQ), H(x)\leq B\\ 
\pi^{-1}(x) \text{ smooth}}}  \omega_\pi^\flat(x,B)^{r} \]  establishes 
\[\frac{\c{M}^\flat_r(\pi, B)}{c_n B^{n+1}}=\mu_r 
\frac{\hat{\sigma}(B)^r}{(\Delta(\pi)\log \log B)^{r/2}}  
+O_{r}
\Big( \frac{1}{\log \log B}\Big)
.\]   
The proof of~\eqref{eq:evenmo} can now be concluded by using~\eqref{eq:shouldspfgy67} to verify 
\begin{align*}
\frac{\hat{\sigma}(B)^r}{(\Delta(\pi)\log \log B)^{r/2}}  
&=\Bigg(1+O_r
\Bigg(\frac{\log \log \log \log B}{\log \log B}\Bigg)\Bigg)^{r/2} 
\\&=1+O_r
\Bigg(\frac{\log \log \log \log B}{\log \log B}\Bigg)
.\end{align*}
The proof of~\eqref{eq:oddmo}  can be 
performed  
in an entirely analogous manner by using~\eqref{eq:oddmoments}. 
\end{proof} 

\subsection{Proof of Theorem~\ref{thm:moment}}
\label{s:bachbwv997}
We first require the following preparatory lemma.

\begin{lemma}
\label{lem:comparelasbux}
Let $y(B),z(B)$ be two functions satisfying 
\[
z(B)>1,
\lim_{B\to\infty} y(B)= 
\infty
\
\text{ and } 
\
\lim_{B\to\infty} 
\frac{\log y(B)}{\log B}=0
.\]
Let $m\in \Z_{\geq 0}$ and let $F\in \Z[x_0,\ldots,x_n]$
be a primitive homogeneous polynomial. Then 
\[
\sum_{\substack{x\in \P^n(\Q)\\H(x)\leq B\\
F(x)\neq 0
}} \Big(z(B)+\sum_{\substack{p\mid F(x)\\p\leq y(B)}}1\Big)^m
\ll_{F,m}
B^{n+1}(z(B)+\log \log y(B))^m
.\]
\end{lemma}
\begin{proof}
It suffices to show that   for every $r\in \Z\cap [0,m]$ we have
\begin{equation} \label{eq:sumaimf} 
\sum_{\substack{x\in \P^n(\Q), H(x)\leq B \\ F(x)\neq 0}} \Big(\sum_{\substack{p\mid F(x)\\p\leq y(B)}}1\Big)^r \ll_{F,r} B^{n+1}(\log \log y(B))^r ,
\end{equation} 
as the result will then easily follow from the binomial theorem. We have \[
\sum_{\substack{x\in \P^n(\Q),
H(x)\leq B 
\\ F(x)\neq 0
}} \Big(\sum_{\substack{p\mid F(x)\\p\leq y(B)}}1\Big)^r
\leq 
\sum_{p_1,\ldots,p_r\leq y(B)} 
\sum_{\substack{x\in \P^n(\Q), H(x)\leq B
\\
1\leq i \leq r \Rightarrow p_i  \mid F(x)
}} 1
.\]
Letting $Q$ be the least common multiple of the primes $p_1,\ldots,p_r$ we see that 
$\omega(Q)\leq r$ and $\mu(Q)^2=1$. Furthermore, for every $Q\in \N$ having these two properties 
there are at most $r^r$ vectors $(p_1,\ldots,p_r)$ with every prime $p_i$ satisfying $p_i\leq y(B)$
and with $Q$ being the least common multiple of the  $p_i$. This is because for each $1\leq i \leq r$ the 
prime $p_i$ must divide $Q$, so the number of available $p_i$ is at most $\omega(Q)\leq r$.
This shows that \[
\sum_{p_1,\ldots,p_r\leq y(B)} 
\sum_{\substack{x\in \P^n(\Q), H(x)\leq B
\\
1\leq i \leq r \Rightarrow p_i  \mid F(x)
}} 1\ll_r
\sum_{\substack{Q\in \N\\ 
\omega(Q)\leq r}}
\mu(Q)^2
\sum_{\substack{x\in \P^n(\Q), H(x)\leq B
\\
Q \mid F(x)
\\
p\mid Q \Rightarrow p\leq y(B)
}} 1
.\] Letting 
\[
\Upsilon =
\{x \in \P^n(\Z/Q\Z): F(x)\equiv 0 \bmod Q\}
\]
we may obtain the following via
Lemma~\ref{lem:techn} 
and following 
similar steps as 
in the proof of Lemma \ref{lem:harmonic},
\[
\#\Upsilon=\prod_{p\mid Q}
\#\{x \in \P^n(\F_p): F(x)= 0 \}
\leq 
\prod_{p\mid Q}
( \deg F \cdot \#\P^{n-1}(\F_p) )
\leq
Q^{n-1} (2\deg F)^{\omega(Q)}
\ll_r
Q^{n-1}
.\]
Noting that the assumption 
$\log y(B)=o(\log B)$ shows that $y(B)\ll_\epsilon B^\epsilon$ for every $\epsilon>0$.
Hence, we have $Q\leq y(B)^r\ll_\epsilon B^\epsilon$
and invoking
Proposition~\ref{prop:sieve} with~\eqref{eq:chronic_infection_pestilence}
we 
obtain\[
\sum_{\substack{x\in \P^n(\Q), H(x)\leq B
\\
Q  \mid F(x)
}} 1\ll_{\epsilon,r}
\frac{Q^{n-1}    }{\#\PP^n(\ZZ/Q\ZZ)} 
 B^{n+1}   
	+  Q^n 
	\left( B+\frac{B^n}{Q^{n}} (\log B)^{[1/n]}
	\right)  
	\ll_{\epsilon,r}
\frac{1   }{Q}  B^{n+1} +B^{n+\frac{1}{10}}
,\] thus 
\[
\sum_{\substack{Q\in \N,  \omega(Q)\leq r\\
p\mid Q \Rightarrow p\leq y(B) 
}} \mu(Q)^2
\sum_{\substack{x\in \P^n(\Q), H(x)\leq B
\\
Q  \mid F(x)
}} 1\ll_{\epsilon,r}
B^{n+1}
\l(
\sum_{\substack{Q\in \N,  \omega(Q)\leq r\\
p\mid Q \Rightarrow p\leq y(B) 
}} \frac{\mu(Q)^2}{Q}\r)
+B^{n+\frac{1}{10}}\l(
\sum_{\substack{Q\in \N,  \omega(Q)\leq r\\
p\mid Q \Rightarrow p\leq y(B) 
}} \mu(Q)^2\r)
.\]
The last sum over $Q$ is at most $y(B)^r       \ll_\epsilon      B^\epsilon$, while the previous satisfies 
\[
\sum_{\substack{Q\in \N,  \omega(Q)\leq r\\
p\mid Q \Rightarrow p\leq y(B) 
}} \frac{\mu(Q)^2}{Q}
\leq 
\l(\sum_{p\leq y(B)}\frac{1}{p}\r)^r
\ll_r
(\log \log y(B))^r.\]  
This verifies~\eqref{eq:sumaimf} 
and thus 
concludes the proof.
\end{proof}
We begin the proof of Theorem~\ref{thm:moment}
by noting that 
\begin{equation}
\label{eq:boundnote}
\c{M}_r(\pi, B)=
\sum_{\substack{ x \in \PP^n(\QQ), H(x)\leq B\\ f(x)\neq 0}}
\l(
\frac{
\omega_\pi(x)
-
\Delta(\pi)\log \log B
}
{\sqrt{\Delta(\pi)\log \log B}}
\r)^{r}
+O_r(B^{n}(\log B)^r)
.\end{equation}
This is because by Lemma~\ref{lem:bounds}
we have 
\[
\sum_{\substack{ x \in \PP^n(\QQ), H(x)\leq B\\ \pi^{-1}(x) \text{ smooth} \\ f(x)=0}}
\l(
\frac{
\omega_\pi(x)
-
\Delta(\pi)\log \log B
}
{\sqrt{\Delta(\pi)\log \log B}}
\r)^{r}
\ll_r
\sum_{\substack{ x \in \PP^n(\QQ) \\ H(x)\leq B \\ f(x) = 0}}
(\log B)^r
\ll B^n
(\log B)^r.
\]
We continue 
the
proof of Theorem~\ref{thm:moment}
by
applying  
Proposition~\ref{prop:applicationofgranvsound}.
For every $x$ in the sum on the right side of~\eqref{eq:boundnote},
Lemma~\ref{lem:bounds} shows that
\begin{align*}
&\omega_\pi(x)-\Delta(\pi) \log \log B \\
&=
\omega_\pi^\flat(x,B)
+\sum_{{p\leq t_0(B)}}\theta_p(x)
+\sum_{{t_1(B)<p \ll B^D}}\theta_p(x)
+
\Big(\Big\{ \hspace{-5pt}
\sum_{t_0(B)<p\leq t_1(B)} \hspace{-5pt} \sigma_p
\Big\}
-\Delta(\pi) \log \log B
\Big)
.\end{align*}
Owing to~\eqref{eq:shouldsp} 
the last term is 
$\ll_{\c{C},\epsilon_1} \log \log \log \log B$.
Using Lemma~\ref{lem:f} and the trivial bound $0\leq \theta_p(x)\leq 1$
we see that  
\[
\sum_{p\leq t_0(B)}\theta_p(x)
=
\sum_{p\leq A}\theta_p(x)
+
\sum_{\substack{p\mid f(x) \\ A<p\leq t_0(B)}}\theta_p(x)
\ll
1+
\sum_{\substack{p \mid f(x)\\ p\leq t_0(B)}} 1
.\]
Observe that 
\begin{equation}
\label{eq:obviousboundisobvious}
m \in \Z\setminus \{0\},
z \in \R_{> 1}
\Rightarrow
\#\{p>z: p\mid m\}
\leq \frac{\log |m|}{\log z}
\end{equation}
because 
$z^{\#\{p>z: p\mid m\}}\leq \prod_{p\mid m} p\leq m$. 
Hence, whenever $x\in \P^n(\Q)$ is such that 
$H(x)\leq B$ and $f(x)\neq 0$ 
we deduce by 
Lemma~\ref{lem:f}
that 
\[
\sum_{p>t_1(B)}\theta_p(x)
=
\sum_{\substack{p|f(x)\\p>t_1(B)}}\theta_p(x)
\leq
\sum_{\substack{p|f(x)\\p>t_1(B)}} 1
\leq 
\frac{\log |f(x)|}{\log t_1(B)}
\ll
  \frac{\log B}{\ \log t_1(B)}
\ll_{\epsilon_1} 
1,\]
where we used
the fact that 
$\log f(x)\ll \log H(x)$. 
We are thus led to
the conclusion that
for any $x$ 
on the right side of~\eqref{eq:boundnote}
we have 
\[
\omega_\pi(x)-\Delta(\pi) \log \log B
=\omega_\pi^\flat(x,B)
+E_\pi(x,B)
,\]
for some function $E_\pi(x,B)$ satisfying
\begin{equation} \label{eqn:E}
|E_\pi(x,B)|\ll_{\c{C},\epsilon_1} 
\log \log \log \log B+
\sum_{\substack{p \mid f(x), p\leq t_0(B)}} 1
.\end{equation}
Therefore, we obtain for $r\neq 0$ that 
\[
(\omega_\pi(x)-\Delta(\pi) \log \log B)^r
=\omega_\pi^\flat(x,B)^r
+\sum_{k=0}^{r-1}
{r \choose k}
\omega_\pi^\flat(x,B)^k
E_\pi(x,B)^{r-k}
.\]
This allows the comparison 
with the ``truncated moment'' \eqref{def:truncated_moment}, to find via~\eqref{eq:boundnote}
that
\[
\big|
\c{M}_r(\pi, B)-
\c{M}_r^\flat(\pi, B)
\big|
\ll_{r}
\frac{B^n}{(\log B)^{-r}}
+
\sum_{k=0}^{r-1}
\sum_{\substack{ x \in \PP^n(\QQ), H(x)\leq B\\ f(x)\neq 0 }} 
\frac{
\omega_\pi^\flat(x,B)^k E_\pi(x,B)^{r-k} }
{(\log \log B)^{r/2}}
.\]
By Cauchy's inequality we see that the last 
sum is 
\[
\ll_{r}
\frac{1}{{(\log \log B)^{r/2}}}
\sum_{k=0}^{r-1} 
\Big(
\sum_{\substack{ x \in \PP^n(\QQ),
H(x)\leq B\\
 f(x)\neq 0
}} 
\omega_\pi^\flat(x,B)^{2k}
\Big)^{1/2}
\Big(
\sum_{\substack{ x \in \PP^n(\QQ), H(x)\leq B
\\  f(x)\neq 0
}} 
E_\pi(x,B)^{2(r-k)}
\Big)^{1/2}.
\]
We apply Proposition~\ref{prop:applicationofgranvsound}
with $r=2k$ to obtain
\[
\big| \c{M}_r(\pi, B)- \c{M}_r^\flat(\pi, B) \big|
 \ll_r
\frac{B^n}{(\log B)^{-r}}+
\sum_{k=0}^{r-1}  
\frac{B^{(n+1)/2}}{(\log \log B)^{(r-k)/2}}
\Big(
\sum_{\substack{ x \in \PP^n(\QQ)
\\
H(x)\leq B
\\  f(x)\neq 0
}} 
\hspace{-0,3cm}
E_\pi(x,B)^{2(r-k)}
\Big)^{1/2}
.\]
Recalling \eqref{eqn:E} and applying Lemma~\ref{lem:comparelasbux}
with
\[
m=2(r-k), \
z(B):=\log \log \log   \log  B
\ \text{ and } \ 
y(B):=t_0(B),
\]
we see that, in light of $z(B)+\log \log y(B)\ll_r
 \log \log \log   \log  B 
$,
one has 
\begin{align*}
\frac{ 
\c{M}_r(\pi, B)-
\c{M}_r^\flat(\pi, B) 
}
{B^{n+1}
} 
\ll_r  
\frac{(\log B)^r}{B}+\frac{\log \log \log \log B 
}{(\log \log B)^{1/2}} 
.\end{align*}
We conclude that 
\[
\c{M}_r(\pi, B)=
\c{M}_r^\flat(\pi, B)+ 
O_r
\Big(B^{n+1} \frac{\log \log \log \log B }{(\log \log B)^{1/2}}\Big)
.\]
An application of 
Proposition~\ref{prop:applicationofgranvsound}
completes
the proof of 
Theorem~\ref{thm:moment}.
\qed

\subsection{Proof of Theorem~\ref{thm:gaussian}} \label{s:bachbwv1080} 
\begin{lemma}\label{lem:fasolada} There exists a set  $\c{S} \subset \{x\in \P^n(\Q):H(x)\leq B\}$
with \begin{equation}\label{eq:fasolena}
\#\{x\in \P^n(\Q):H(x)\leq B, x\notin \c{S}\}=O\Big(\frac{B^{n+1}}{(\log \log \log B)^2}\Big)\end{equation}
and such that for all $x\in \c{S}$ we have 
\begin{equation}\label{eq:fasoldio}
\frac{\omega_\pi(x)-\Delta(\pi) \log \log B}{\sqrt{\Delta(\pi) \log \log B}}
=
\frac{\omega_\pi(x)-\Delta(\pi) \log \log H(x)}{\sqrt{\Delta(\pi) \log \log H(x)}}
+O\Big(\frac{1}{\sqrt{\log B}}\Big)
.\end{equation}
\end{lemma}
\begin{proof} Denote $\Xi(B)=\log \log \log B$ 
and define 
\[\c{S}:=\Bigg\{x\in \P^n(\Q): \frac{B}{\log B}<H(x)\leq B, \pi^{-1}(x)  \text{ smooth},
\Bigg|\frac{\omega_\pi(x)-\Delta(\pi) \log \log B}{\sqrt{\Delta(\pi) \log \log B}}\Bigg|  \leq \Xi(B)
\Bigg\}.\] 
The cardinality of those 
$x\in \P^n(\Q)$ with $H(x)\leq B, \pi^{-1}(x)$ smooth and 
\[\Bigg|\frac{\omega_\pi(x)-\Delta(\pi) \log \log B}{\sqrt{\Delta(\pi) \log \log B}}\Bigg|>\Xi(B)
\]
is at most 
\[\sum_{\substack{ x \in \PP^n(\QQ), H(x)\leq B\\ \pi^{-1}(x) \text{ smooth}}}
\l(
\frac{
\omega_\pi(x)
-
\Delta(\pi)\log \log B
}
{\Xi(B)\sqrt{\Delta(\pi)\log \log B}}
\r)^{2}
=\frac{c_n B^{n+1}}{\Xi(B)^2}(1+o(1))\ll \frac{B^{n+1}}{\Xi(B)^2},\]
where the case $r=2$ of Theorem~\ref{thm:moment} has been used.
This provides us with
\[\#\c{S}=c_n B^{n+1}+O\Big(\frac{B^{n+1}}{\Xi(B)^2}\Big).\]
Now note that for all $x\in \P^n(\Q)$ with  $B/\log B<H(x)\leq B$ we have 
\[\log \log H(x)= \log \log B+O\Big(\frac{\log \log B}{\log B}\Big)
,\] 
therefore $(\log \log H(x)/\log \log B)^{1/2}=1+O(1/\log B)$. Thus, for each such $x$ we get
\[
\frac{\omega_\pi(x)-\Delta(\pi) \log \log B}{\sqrt{\Delta(\pi) \log \log B}}
= \frac{\omega_\pi(x)-\Delta(\pi) \log \log B}{\sqrt{\Delta(\pi) \log \log H(x)}} +
O\Bigg( \frac{|\omega_\pi(x)-\Delta(\pi) \log \log B|}{\log B\sqrt{\log \log B}} \Bigg)
.\]
We deduce that  if $x\in \c{S}$ then this is 
\[\frac{\omega_\pi(x)-\Delta(\pi) \log \log B}{\sqrt{\Delta(\pi) \log \log H(x)}} + O\bigg( \frac{\Xi(B)}{\log B} \bigg)
=\frac{\omega_\pi(x)-\Delta(\pi) \log \log H(x)}{\sqrt{\Delta(\pi) \log \log H(x)}} + O\bigg( \frac{\sqrt{\log \log B}}{\log B}+\frac{\Xi(B)}{\log B} \bigg),\]
which is sufficient for our lemma.
\end{proof}
We are now in place to prove Theorem~\ref{thm:gaussian}. 
 For $z\in \R$ we 
denote the distribution function of the 
standard normal distribution  
by  \[\Phi(z):=\frac{1}{\sqrt{2\pi}} \int_{-\infty}^z \mathrm{e}^{-\frac{t^2}{\!2}} \mathrm{d}t .\]
Recall the definition of the probability measure $\nu_B$ in~\eqref{eq:puppenkonigEPI} and 
note that the set $\P^n(\Q)$ becomes a probability space once equipped with the measure $\nu_B$. (The measure $\nu_B$ is supported on the rational points of height at most $B$.)
For any $B\in \R_{\geq 3}$ we consider the random variable $ \texttt{X}_B$ defined on   $\P^n(\Q)$   as follows,
\[ \texttt{X}_B(x):=\begin{cases} \frac{\omega_\pi(x)-\Delta(\pi) \log \log B}{\sqrt{\Delta(\pi) \log \log B}}, 
& \pi^{-1}(x) \text{ smooth},
\\  0, & \text{otherwise.} \end{cases}\]
For $r\in \Z_{\geq 0}$ the $r$-th moment of $ \texttt{X}_B$ is by definition equal to 
\[\int_{-\infty}^{+\infty} \texttt{X}_B^r \mathrm{d}\nu_B=
\sum_{\substack{x\in \P^n(\Q)\\\pi^{-1}(x) \text{ smooth}
}} \Big(\frac{\omega_\pi(x)-\Delta(\pi) \log \log B}{\sqrt{\Delta(\pi) \log \log B}}\Big)^r 
\frac{\nu_B(\{x\})}{\#\{x\in \P^n(\Q):H(x)\leq B\}} 
\] and recalling~\eqref{def:momer} 
we see that this coincides with  $ \c{M}_r(\pi,B) / \#\{x\in \P^n(\Q):H(x)\leq B\}$. 
Theorem~\ref{thm:moment} shows that 
\[
\lim_{B\to+\infty} \int_{-\infty}^{+\infty} \texttt{X}_B^r \mathrm{d}\nu_B
\]
exists and is equal to the $r$-th moment of the standard normal distribution. By~\cite[Th. 30.2]{MR0466055} we get that 
$\texttt{X}_B$ converges in law to the standard normal distribution, i.e. for every  $y \in \R$ we have 
\begin{equation}\label{lem:fasoladabitte}  \lim_{B\to+\infty} \nu_B \left( \left\{x \in \PP^n(\Q): \texttt{X}_B  \leq y \right\}\right) =\Phi(y)
.\end{equation}

Next, for every fixed $\epsilon>0,z\in \R$ and  all sufficiently large $B$ we see that the error term in~\eqref{eq:fasoldio}
has modulus at most $\epsilon$, therefore~\eqref{eq:fasoldio} gives
\begin{equation}\label{lem:fasolada1} 
\nu_B \left( \left\{x \in \c{S}: \texttt{X}_B(x)  \leq z-\epsilon \right\}\right) \leq  \nu_B \left(
\left\{x \in \c{S}: \frac{\omega_\pi(x)-\Delta(\pi) \log \log H(x)}{\sqrt{\Delta(\pi) \log \log H(x)}}
 \leq z
\right\}\right)
\end{equation}
and  \begin{equation}\label{lem:fasolada2} 
\nu_B \left( \left\{x \in \c{S}: \texttt{X}_B(x)  \leq z+\epsilon \right\}\right) \geq  \nu_B \left( \left\{x \in \c{S}: \frac{\omega_\pi(x)-\Delta(\pi) \log 
\log H(x)}{\sqrt{\Delta(\pi) \log \log H(x)}}  \leq z\right\}\right)\end{equation}
for all sufficiently large $B$.
Observe that the set $\c{S}$ in Lemma~\ref{lem:fasolada} 
satisfies the following  as $B\to+\infty$ due to~\eqref{eq:fasolena}, 
\begin{equation}\label{lem:fasolada0} 
\nu_B(\c{S})=1+O\Big(\frac{1}{(\log \log \log B)^2}\Big)=1+o(1)
.\end{equation}
In light of~\eqref{lem:fasoladabitte} this means that $ \lim_{B\to+\infty} \nu_B\left(\left\{x \in \c{S}: \texttt{X}_B  \leq y\right\}\right)=\Phi(y),$ 
which, when applied to $y=z-\epsilon$ and $y=z+\epsilon$ and combined with~\eqref{lem:fasolada1} and~\eqref{lem:fasolada2},
yields
\[ 
 \liminf_{B\to+\infty}\nu_B \left(
\left\{x \in \c{S}: \frac{\omega_\pi(x)-\Delta(\pi) \log \log H(x)}{\sqrt{\Delta(\pi) \log \log H(x)}}
 \leq z
\right\}\right)
\geq 
\Phi(z-\epsilon)
\]
and 
\[\limsup_{B\to+\infty}\nu_B \left(
\left\{x \in \c{S}: \frac{\omega_\pi(x)-\Delta(\pi) \log \log H(x)}{\sqrt{\Delta(\pi) \log \log H(x)}}
 \leq z
\right\}\right)
\leq \Phi(z+\epsilon) 
.\]
Letting $\epsilon\to 0$  and 
using the fact that $\Phi$ is continuous
we obtain \[\lim_{B\to+\infty}\nu_B \left(
\left\{x \in \c{S}: \frac{\omega_\pi(x)-\Delta(\pi) \log \log H(x)}{\sqrt{\Delta(\pi) \log \log H(x)}}
 \leq z
\right\}\right)
= \Phi(z) 
,\]
which, by~\eqref{lem:fasolada0} implies that    \[\lim_{B\to+\infty}\nu_B \left( \left\{x \in \P^n(\Q): \frac{\omega_\pi(x)-\Delta(\pi) \log \log H(x)}
{\sqrt{\Delta(\pi) \log \log H(x)}}  \leq z \right\}\right) = \Phi(z)  .\] Since this holds for every fixed $z\in \R$ it gives Theorem~\ref{thm:gaussian} as an 
immediate consequence. \qed

\section{The pseudo-split case}
\label{sec:ps}

\subsection{Proof of Theorem \ref{thm:Delta=0}}
\label{s:existe}
We let $\pi:V \to \PP^n$ be as in Theorem \ref{thm:Delta=0} and choose a model for $\pi$
as in \S \ref{s:equidistribution}.

\subsubsection{Existence of the limit}
We first prove the existence of \eqref{def:jjlimit} using the versions of the sieve of Ekedahl given
in \cite[\S 4.1]{LS16}, \cite[\S 3]{BBL16} and  \cite[\S 3]{Bha14}.
We begin with a strengthening of Lemma \ref{lem:f}. (Here it is crucial that $\Delta(\pi) = 0$.)

\begin{lemma} \label{lem:Z}
	There exists a closed subset $Z \subset \PP_\ZZ^n$ of codimension at least $2$ and a constant $A > 0$
	with the following property. Let $p > A$ be a prime and $x \in \PP^n(\Z_p)$ such that 
	$\pi^{-1}(x)(\Z_p) = \emptyset$. Then $x \bmod p \in Z(\FF_p)$.
\end{lemma}
\begin{proof}
	This is a special case of \cite[Prop.~4.1]{LS16}.
\end{proof}

\begin{lemma} \label{lem:Bhargava}
Let $A$ be as in Lemma \ref{lem:Z} Then for every $B,M> 1$ 
we have 
\[
\#\left\{x\in \P^n(\Q):
	\begin{array}{l}
		H(x)\leq B, \pi^{-1}(x) \text{ smooth},  \\
		\exists p>M \text{ s.t. } \pi^{-1}(x)(\Q_p)= \emptyset
		\end{array}	\right\}
\ll
\frac{B^{n+1}}{M \log M} +  B^n
,\]
where the implied constant depends at most on $A$ and $\pi$.
\end{lemma}
\begin{proof}
	This follows immediately from Lemma \ref{lem:Z} and Bhargava's effective version of the Ekedahl sieve \cite[Thm.~3.3]{Bha14}.
\end{proof}
We now prove a strengthening of Proposition \ref{prop:arcturus} in the case $\Delta(\pi) = 0$. The crucial point about the next proposition is that it gives an asymptotic formula for a counting problem which has local conditions imposed at \emph{every} place $v$, whereas Proposition \ref{prop:arcturus} only imposes conditions at \emph{finitely} many primes. In what follows we use the measures $\vartheta_v$ from \S \ref{sec:mu_p}.

\begin{proposition} \label{prop:Ekedahl} 
	Let $S$ be a finite set of places of $\Q$. Then
	$$
	\lim_{B \to \infty} \frac{
	\#\left\{x\in \P^n(\Q):
	\begin{array}{l}
		H(x)\leq B, \pi^{-1}(x) \text{ smooth},  \\
		v \in S
		\!\Leftrightarrow \! 
		\pi^{-1}(x)(\Q_v)= \emptyset 
		\end{array}	\right\} }
		{\#\{x\in \P^n(\Q):H(x)\leq B\}}
	=   \begin{array}{l}
		\prod_{v \in S} \vartheta_v(\PP^n(\QQ_v) \setminus \pi(V(\QQ_v)) \\
		\quad \times \, \prod_{v \notin S}\vartheta_v(\pi(V(\QQ_v)),
		\end{array}
	$$
	where the right hand side is a convergent Euler product.
	
	Moreover, let $Q \in \NN$ be square-free. Then there exists $K_{0}>0$
	such that 
	\begin{equation} \label{eq:leb} 
	\hspace{-6pt}
	\#\left\{x\in \P^n(\Q):
	\hspace{-4pt}
	\begin{array}{l}
		H(x)\leq B, \pi^{-1}(x) \text{ smooth},\\
		 \pi^{-1}(x)(\Q_p)= \emptyset \, \forall \,p \mid Q
		\end{array} \hspace{-5pt}	\right\} 
	     \ll
	K_0^{\omega(Q)}
	\left(
	\frac{B^{n+1}}{Q^{2}}  + 
	BQ^{n-1} + \frac{B^n \log B}{Q} \right).
	\end{equation}
\end{proposition}
\begin{proof}
	The asymptotic formula is proved using an adaptation of \cite[Thm.~3.8]{BBL16},
	via the version of the sieve of Ekedahl given in \cite[Prop.~3.4]{BBL16}.
	That the condition (3.5) from \cite[Prop 3.4]{BBL16} is satisfied follows from
	Lemma \ref{lem:Z} and \cite[Lem.~3.5]{BBL16}.  
	If $\pi(V(\QQ_v)) \neq \emptyset$ then \cite[Lem.~3.9]{BBL16} implies that $\pi(V(\QQ_v))$
	is measurable, has positive measure and has boundary of measure $0$. Moreover $\pi(V(\QQ_v))
	\subset \PP^n(\QQ_v)$ is closed as $\pi$ is proper.
	It follows that $\PP^n(\QQ_v) \setminus \pi(V(\QQ_v))$
	is measurable, has boundary of measure $0$, and has positive measure if non-empty.
	Therefore the measurability hypotheses in \cite[Prop.~3.4]{BBL16} are all satisfied.
	Applying \cite[Prop.~3.4]{BBL16} gives the asymptotic formula.
	(Note that \cite[Lem.~4.8]{BBL16} works with the Haar measure on $\mathbb{Q}_v^{n+1}$,
	whereas in the statement we use the measure $\vartheta_v$. One easily obtains
	our statement using Lemma \ref{lem:Haar}.)
	
	Next, by the Lang--Weil estimates 
	there exists $K_{0}>0$ such that 
	$
	\#Z(\F_p)\leq~K_{0}~p^{n-2}
	$
	for all $p$.
	Therefore 
	we have $
	\# Z(\ZZ/Q\ZZ) \leq K_{0}^{\omega(Q)} Q^{n-2}
	.$
	Lemma \ref{lem:Z} and Proposition~\ref{prop:sieve}
	now show that
	the left side of~\eqref{eq:leb} 
	is
	\begin{align*}
		&\ll
		\#\left\{x\in \P^n(\Q):
		\begin{array}{l}
		H(x)\leq B, \pi^{-1}(x) \text{ smooth},  \\
		p \mid Q \Rightarrow x \bmod p \in Z(\FF_p)  
		\end{array}	\right\}  \\
		& \ll \frac{\# Z(\ZZ/Q\ZZ)}{\# \PP^n(\ZZ/Q\ZZ)} B^{n+1}  + Q\#Z(\ZZ/Q\ZZ)\left( B+ \frac{B^n}{Q^{n}} \log B\right) \\
		& \ll  K_0^{\omega(Q)}Q^{-2} B^{n+1}  + K_0^{\omega(Q)} Q^{n-1}B   + K_0^{\omega(Q)}Q^{-1} B^n \log B. \qedhere
	\end{align*}
\end{proof}
We now show the existence of the limit \eqref{def:jjlimit}. In fact, we prove the following explicit formula. (Recall the definition of $\tau_\pi(B,j)$ from~\eqref{def:fpijb}.)

\begin{proposition} \label{prop:limit}
	We have
 \[
	\tau_\pi(j) : =\lim_{B\to\infty}
	\tau_\pi(B,j)=
	\sum_{p_1<p_2<\cdots<p_j} \prod_{p\mid p_1\cdots p_j}
	\vartheta_p(\PP^n(\QQ_{p}) \setminus 	\pi(V(\QQ_{p})) 
	\prod_{p\nmid p_1\cdots p_j} \vartheta_p(\pi(V(\QQ_p)),
 \]
 where the sum and products are convergent.
\end{proposition}
\begin{proof}
If $Q\in [1,B^{1/3})$ is a square-free integer 
then one can immediately see from~\eqref{eq:leb}
that 
\begin{equation}
	\label{eq:leb3} \#\left\{x\in \P^n(\Q):
	\begin{array}{l}
		H(x)\leq B, \pi^{-1}(x) \text{ smooth},  \\
		 p \mid Q  \Rightarrow \pi^{-1}(x)(\Q_p)= \emptyset
		\end{array}	\right\} 
	\ll 
	\frac{K_{0}^{\omega(Q)}}{Q^2} 
	B^{n+1}
		.\end{equation}
Combining the upper bound and the asymptotic 
provided
by Proposition~\ref{prop:Ekedahl} 
one sees that for every square-free $Q\neq 0$ we have
\begin{equation}
\label{eq:lute}
\prod_{p \mid Q} \vartheta_p(\PP^n(\QQ_p) \setminus \pi(V(\QQ_p)) \prod_{p \nmid Q}\vartheta_p(\pi(V(\QQ_p))
\ll  \frac{K_{0}^{\omega(Q)}}{Q^2} 
,\end{equation}
with an implied constant independent of $Q$.
Fix any $M>1$.
By Lemma~\ref{lem:Bhargava}
we see that 
\[
\tau_\pi(B,j)
=\frac{
\#\left\{x\in \P^n(\Q):
	\begin{array}{l}
		H(x)\leq B, \pi^{-1}(x) \text{ smooth}, 		\omega_\pi(x)=j,\\
		 \pi^{-1}(x)(\Q_p)= \emptyset \Rightarrow p\leq M
		\end{array}	\right\}}
	{\#\{x\in \P^n(\Q):H(x)\leq B\}}
+O\left(\frac{1}{M \log M} + \frac{1}{B}\right)
,\]
with an implied constant that is independent of $j,M$ and $B$.
We infer that  
$\tau_\pi(B,j)$
equals
\[
\sum_{p_1<p_2<\cdots<p_j\leq M}
		\frac{
		\#\left\{x\in \P^n(\Q):
	\begin{array}{l}
		H(x)\leq B, \pi^{-1}(x) \text{ smooth},  \\
		 \pi^{-1}(x)(\Q_p)= \emptyset 
		\!\Leftrightarrow \! 
		p\mid p_1\cdots p_j
		\end{array}	\right\}}
		{\#\{x\in \P^n(\Q):H(x)\leq B\}}
+ O\left(\frac{1}{M \log M} + \frac{1}{B}\right)
.\]
Fixing the value of $M$ and taking the limit as $B\to\infty$,
we see from Proposition~\ref{prop:Ekedahl} that 
\begin{equation}
\label{eq:coltrtne}
\limsup_{B\to\infty}
\Bigg|
\tau_\pi(B,j)-
\hspace{-0,3cm}
\sum_{p_1<\cdots<p_j\leq M}
		\prod_{p\mid p_1\cdots p_j}\vartheta_p(\PP^n(\QQ_{p}) \setminus \pi(V(\QQ_{p})) 
		\prod_{p\nmid p_1\cdots p_j} \vartheta_p(\pi(V(\QQ_p))
\Bigg| 
\ll \frac{1}{M}
.\end{equation}
Note that the infinite series  
\[
\sum_{p_1<p_2<\cdots<p_j}\prod_{p\mid p_1\cdots p_j}\vartheta_p(\PP^n(\QQ_{p}) \setminus \pi(V(\QQ_{p})) 
		\prod_{p\nmid p_1\cdots p_j} \vartheta_p(\pi(V(\QQ_p))
\]
converges owing to the bound 
\begin{equation}
\label{eq:can}
\prod_{p\mid p_1\cdots p_j}\vartheta_p(\PP^n(\QQ_{p}) \setminus \pi(V(\QQ_{p})) 
		\prod_{p\nmid p_1\cdots p_j} \vartheta_p(\pi(V(\QQ_p))
		\ll \frac{K_{0}^j}{(p_1\cdots p_j)^2}
,		\end{equation}
that follows from~\eqref{eq:lute}.
Taking $M$ to be arbitrarily large 
in~\eqref{eq:coltrtne}
proves the result.
\end{proof}

\subsubsection{Probability measure}
\label{s:koopmangod}
We now show that \eqref{def:jjlimit} indeed defines a probability measure on $\ZZ$. To do so,
it suffices to show that
\begin{equation}
\label{eq:space}
\sum_{j \in \ZZ} \tau_\pi(j) = 1.
\end{equation}
 Partitioning all possible values for $\omega_\pi(x)$ we have
\[ \#\{x\in \P^n(\Q):H(x)\leq B, \pi^{-1}(x) \text{ smooth}\}
	=\sum_{j=0}^\infty
	\#\left\{x\in \P^n(\Q):
\begin{array}{l}
	H(x)\leq B, \omega_\pi(x)=j, \\
	\pi^{-1}(x) \text{ smooth}
	\end{array}	\right\}
.\]
Fix $j_0 \in \N$ and note that if $\omega_\pi(x)>j_0$ then the largest prime $p$ such that 
$\pi^{-1}(x)(\Q_p)=\emptyset$ exceeds the $j_0$-th largest prime, therefore it is at least $j_0$.
This shows that 
\begin{align*}
	\sum_{j > j_0}
	&\#\left\{x\in \P^n(\Q):
	\begin{array}{l}
	H(x)\leq B, \omega_\pi(x)=j, \\
	\pi^{-1}(x) \text{ smooth}
	\end{array}	\right\}
	\\ \leq  &
	\#\left\{x\in \P^n(\Q):
	\begin{array}{l}
	H(x)\leq B, \pi^{-1}(x) \text{ smooth},  \\
	\exists p>j_0 \text{ s.t. } \pi^{-1}(x)(\Q_p)= \emptyset
	\end{array}	\right\}
\end{align*}
which is $O(B^{n+1} j_0^{-1} + B^n)$ by Lemma~\ref{lem:Bhargava}.
Dividing  by $\#\{x\in \P^n(\Q):H(x)\leq B\}$ gives  
\[
\frac{\#\{x\in \P^n(\Q):H(x)\leq B, \pi^{-1}(x) \text{ smooth}\}}{\#\{x\in \P^n(\Q):H(x)\leq B\}}
	=\sum_{0\leq j \leq j_0} \tau_\pi(B,j)
	+O\l(\frac{1}{j_0} + \frac{1}{B}\r)
,\]
with an implied constant that is independent of $j_0$.
Letting $B\to\infty$ we obtain
\[\sum_{0\leq j \leq j_0} \tau_\pi(j) = 1 +O\l(\frac{1}{j_0}\r).\]
Letting $j_0\to\infty$ we infer that the sum over $j$ converges to $1$,
thus verifying~\eqref{eq:space}.  
\subsubsection{Upper bounds}
\label{s:upperfor13}
We now  prove~\eqref{eq:wowcan}.
Combining~\eqref{eq:can} and Proposition~\ref{prop:limit} shows that 
\[ 
\tau_\pi(j) 
=
\lim_{B\to\infty}
\tau_\pi(j,B) 
\ll_\pi 
K_{0}^j
\sum_{p_1<\ldots<p_j} 
\frac{1}{p_1^{2}
\cdots p_j^{2}} 
.\]
Let us denote
the primes in ascending order as 
$q_1=2,q_2=3,$ etc.
Writing 
\[
K_{0}^j
\sum_{p_1<\ldots<p_j} 
\frac{1}{p_1^{2}
\cdots p_j^{2}} 
=
\sum_{p_1\geq 2} \frac{K_{0}}{p_1^{2}}
\sum_{p_2>p_1} \frac{K_{0}}{p_2^{2}}
\cdots
\sum_{p_j>p_{j-1}} \frac{K_{0}}{p_j^{2}}
,\]
we observe that 
the sum over $p_2$ contains all primes $p\geq q_2$,
the sum over $p_3$ contains all primes $p\geq q_3$
and so on.
Therefore,
one has  
\[
\tau_\pi(j) 
\ll
\prod_{i=1}^j
\l(
\sum_{p\geq q_i} \frac{K_{0}}{p^2}
\r)
.\]
By the prime number theorem and partial summation we 
see that  
$
\sum_{p\geq z}p^{-2} \leq c_0/(z\log z)
$
for some absolute $c_0>1$,
thus  
$\tau_\pi(j)
\prod_{i=1}^j
(q_i \log q_i)
\ll
(c_0 K_{0})^j
$.  
Using $q_i\sim i \log i$
and 
the prime number theorem
with partial summation we obtain
\begin{align*}
\log \l(
\prod_{i=1}^j q_i \log q_i
\r)
=
&
\sum_{p\leq q_j} \log p
+
\sum_{p\leq q_j} \log \log p 
\\
=&q_j +O\l(\frac{q_j  }{\log q_j}\r)
+
\frac{q_j  \log \log q_j}{\log q_j} +O\l(\frac{q_j  \log \log q_j}{(\log q_j)^2}\r)
\\
=&
j \log j+ j \log \log j+ o(j \log \log j ) 
\\
\geq 
&
j \log j +\frac{3}{4}
j \log \log j,
\end{align*}
for all sufficiently large $j$.
We deduce that  for all large $j$ one
has 
\[\tau_\pi(j)
\leq
(c_0 K_{0})^j
\prod_{i=1}^j
\frac{1}{q_i \log q_i}
\leq 
 \frac{(c_0 K_0 )^j}
{ j^j
(\log j)^{\frac{3j}{4}}}
\leq 
 \frac{1}{j^j(\log j)^{j/2}}
,\]
from which~\eqref{eq:wowcan}
follows. This completes the proof of Theorem \ref{thm:Delta=0}. \qed

\subsection{The family of diagonal cubic surfaces} \label{sec:cubics}
We now return to Example \ref{ex:cubics}, and prove the claim that 
there exists an absolute constant $c>0$ such that $\tau_\pi(j) >  c (1+j)^{-3j}$.

Let $y = (y_0:y_1:y_2:y_3) \in \PP^3(\ZZ_p)$ where $(y_0,\dots,y_3)$ is primitive. By the criterion in~\cite[p.28]{MR927558},
if a prime $p\equiv 1\bmod{3}$ satisfies $p\nmid y_0y_1$,
$p\|y_2$,
$p \|y_3$
and neither $-y_1/y_0$ nor $-y_3/y_2$ are cubes,
then  $\pi^{-1}(y)$
has no $p$-adic point.
It is easy to see that there exists an absolute constant $K_1>0$ such that 
the measure of this with respect to $\vartheta_p$ is at least $K_1p^{-2}$. Hence, denoting by
$q_i$ the $i$-th largest prime being $1\bmod{3}$,
Proposition~\ref{prop:limit} gives
$$\tau_\pi(j) \geq  K_2\prod_{i=1}^j \frac{K_1}{q_i^{2}},$$
for some constant $K_2>0$ (the product in  Proposition~\ref{prop:limit}  being convergent). By the prime number theorem for arithmetic progressions we have $q_i \sim 2 i \log i $, hence
\[
\log \tau_\pi(j)^{-1}
\leq -\log K_2
-j (\log K_1)+2 \hspace{-30pt}
\sum_{\substack{p\leq (2 j \log j)(1+o(1))\\p\equiv 1 \bmod 3}}
\hspace{-30pt}
\log p
=(2 j \log j)(1+o(1))
\leq 3j \log j
\]
for all sufficiently large $j$.
This proves the claim. \qed

\subsection{Proof of Theorem \ref{thm:Ax-Kochen}} \label{sec:Ax-Kochen}
The implication $\Leftarrow$ is clear. For the other, assume that 
$V(\QQ_p) \to \PP^n(\QQ_p)$ is not surjective for infinitely many primes $p$. Let $S$
be a finite set of such primes and let $x_p \in \PP^n(\QQ_p) \setminus \pi(V(\QQ_p))$ for $p \in S$.
By Proposition \ref{prop:sieve}, a positive proportion of $x \in \PP^n(\QQ_p)$ 
are arbitrarily close to the $x_p$ for all $p \in S$. Moreover, as $\pi$ is proper the set $\pi(V(\QQ_p))$
is closed with respect to the $p$-adic topology. It follows that provided the $x$ are sufficiently close to the $x_p$ we have
$\pi^{-1}(x)(\QQ_p) = \emptyset$ for all $p \in S$. Hence for such $x$ we have $\omega_\pi(x) \geq \#S$. As $S$ can be chosen
sufficiently large, the result follows. \qed

\subsection{Proof of Theorem \ref{thm:Hasse}}
As shown in the proof of Proposition \ref{prop:Ekedahl}, the sets $\pi(V(\QQ_v))$ and $\PP^n(\QQ_v) \setminus \pi(V(\QQ_v))$ are measurable with respect to $\vartheta_v$, and have positive measure if non-empty. The result now 
follows as the Euler product in Proposition~\ref{prop:Ekedahl} is convergent. \qed

\subsection{Proof of Theorem~\ref{thm:Deltazeromoments}}
\label{s:giftmoment}
Let $\c{N}_{r}(\pi,B)$ be as in \eqref{def:N_moment}. We begin with the following.

\begin{lemma} 
\label{lem:supexp2}
For every $r\in \Z_{\geq 0}$ we have 
$
\c{N}_{r}(\pi,B)
\ll_{r}
B^{n+1}
.$
\end{lemma}
\begin{proof}
Recall $f$ and $A$ from Lemma~\ref{lem:f}.
By Lemma \ref{lem:bounds} we have
$$\c{N}_{r}(\pi,B)
= \sum_{\substack{ x \in \PP^n(\QQ), H(x)\leq B\\ f(x) \neq 0}} \omega_\pi(x)^{r}
+ O\left(B^n(\log B)^r\right),$$
where we have used the evident bound $\#\{ x \in \PP^n(\QQ): H(x) \leq B, f(x) = 0\} \ll B^{n}$.
For any $\epsilon>0$ and any $x$ with $f(x) \neq 0$ we have via~\eqref{eq:obviousboundisobvious}
that
\[
\#\{p>B^\epsilon:
\pi^{-1}(x)(\Q_p)=\emptyset
\}
\leq 
\#\{p>B^\epsilon:p \text{ divides } f(x)\}\leq 
\frac
{\log |f(x)| }
{\log (B^\epsilon) }
.\]
As $H(x)\leq B$ implies  $|f(x)|\ll B^{\deg(f)}$, we thus find that
\[
\omega_\pi(x)=
O_\epsilon(1)
+\sum_{A<p\leq B^\epsilon}\theta_p(x)
.\]
Let us now define 
$
\epsilon(r)
:=(3r)^{-1}
$.
Then
\[
\c{N}_{r}(\pi,B) =
\hspace{-10pt}
\sum_{\substack{ x \in \PP^n(\QQ), H(x)\leq B\\ f(x) \neq 0}} 
\l(
O_{r}(1)
+
\hspace{-10pt}
\sum_{A<p\leq B^{\epsilon(r)}
}\hspace{-10pt}\theta_p(x)
\r)^r
\ll_{r}
\sum_{m=0}^r
\sum_{\substack{ x \in \PP^n(\QQ), H(x)\leq B\\ f(x) \neq 0}} 
\l(\sum_{A<p\leq B^{\epsilon(r)}}\hspace{-10pt} \theta_p(x)\r)^m
.\]
To prove the lemma it therefore suffices to show that
\[
\sum_{\substack{ x \in \PP^n(\QQ), H(x)\leq B\\ f(x) \neq 0}} 
\l(\sum_{A<p\leq B^{\epsilon(r)}}\theta_p(x)\r)^m
\ll_{m}
B^{n+1}
.\]
Using the multinomial theorem the sum over $x$ above equals
\[  
\sum_{
\substack{
(m_p)_{p\in (A,B^{\epsilon(r)}]}
\\ 
\sum_{p\in (A,B^{\epsilon(r)}]}m_p=m
}}
\frac{m!}{\prod_{p\in (A,B^{\epsilon(r)}]}m_p!
} 
\sum_{\substack{ x \in \PP^n(\QQ), H(x)\leq B\\ f(x) \neq 0}} 
\prod_{\substack{p\in (A,B^{\epsilon(r)}]\\m_p\neq 0}}\theta_p(x)
.\]
Letting $k$ be the number of $p\in (A,B^{\epsilon(r)}]$
with
$m_p\neq 0$,
shows that the last quantity is 
\[
\ll_m
\sum_{k=1}^m
\sum_{A<p_1<p_2<\cdots<p_k\leq B^{\epsilon(r)}}
\sum_{\substack{ x \in \PP^n(\QQ), H(x)\leq B\\ f(x) \neq 0}} 
\prod_{i=1}^k
\theta_{p_i}(x)
.\] 
By~\eqref{eq:leb3}
and the fact that $k\leq m\leq r$
and 
$\epsilon(r)\leq (3m)^{-1}$
we see that the inner 
sum over $x$ is 
$
\ll
B^{n+1}
\prod_{i=1}^k 
(K_0 p_i^{-2})
$.
We obtain that 
\[
\sum_{A<p_1<p_2<\cdots<p_k\leq B^{\epsilon(r)}}
\sum_{\substack{ x \in \PP^n(\QQ), H(x)\leq B\\ f(x) \neq 0}} 
\prod_{i=1}^k
\theta_{p_i}(x)
\ll
B^{n+1}
\hspace{-10pt}
\sum_{A<p_1<p_2<\cdots<p_k\leq B^{\epsilon(r)}}
\prod_{i=1}^k 
(K_0 p_i^{-2})
\ll B^{n+1}
,\]  
thus concluding our proof.
\end{proof}

Now observe that
\begin{equation}
\label{eq:partfg}
\frac{\c{N}_r(\pi, B)}{\#\{x\in \P^n(\Q):H(x)\leq B\}}
=\sum_{j=0}^\infty
j^r
\tau_\pi(B,j).
\end{equation}
However, for any $M>1$ we may use the inequality 
$
 \mathbf{1}_{\{\omega_\pi(x)>M\}}(x)
\leq 
\omega_\pi(x)/M$
to find that
\begin{equation}
\label{eq:can1}
\sum_{j>M}
j^r
\tau_\pi(B,j)
=
\frac{\sum_{\substack{ x \in \PP^n(\QQ), H(x)\leq B\\ \pi^{-1}(x) \text{ smooth}
}} \omega_\pi(x)^{r} 
 \mathbf{1}_{\{\omega_\pi(x)>M\}}(x)}
 {\#\{x\in \P^n(\Q):H(x)\leq B\}}
\ll
\frac{\c{N}_{r+1}(\pi,B)}{M B^{n+1} }.
\end{equation} 
We now infer from~\eqref{eq:partfg}, \eqref{eq:can1}
and Lemma~\ref{lem:supexp2}
that 
\begin{equation}
\label{eq:hmm}
\left|
\frac{\c{N}_r(\pi, B)}{\#\{x\in \P^n(\Q):H(x)\leq B\}}
-\sum_{0\leq j \leq M}
j^r
\tau_\pi(B,j)
\right|
\ll_{r}
\frac{1}{M}
,\end{equation}
where the implied constant is independent of $B$ and $M$.
Fixing $M$ and taking $B\to \infty$ we are led to the conclusion that 
\[
\limsup_{B\to\infty}
\left|
\frac{\c{N}_r(\pi, B)}{\#\{x\in \P^n(\Q):H(x)\leq B\}}
-\sum_{0\leq j \leq M}
j^r
\tau_\pi(j)
\right|
\ll \frac{1}{M}
.\]
By~\eqref{eq:wowcan} the sum over $j$ is convergent as $M\to\infty$,
which completes the proof. \qed

\subsection{Generalisations} \label{sec:generalisations}
One can consider variants of the function $\omega_\pi$ from \eqref{def:omega}, by considering
real solubility or by dropping conditions at finitely many primes. Namely, let $S$ be a finite set of places
of $\Q$. Then we define
$$\omega_{\pi,S}(x) :=\#\big\{\text{places } v \notin S:\pi^{-1}(x)(\Q_v)=\emptyset\big\}.$$
We have considered the case $\omega_\pi = \omega_{\pi,\infty}(x)$ for simplicity of exposition,
but a minor variant of our arguments yields the following generalisation of Theorem \ref{thm:Delta=0}
(the important point being that the asymptotic in Proposition \ref{prop:Ekedahl} applies to arbitrary $S$).

\begin{theorem} \label{thm:Delta=0_S}
Let $V$ be a smooth projective variety over $\QQ$ equipped with a dominant morphism $\pi: V \to \PP^n$
with geometrically integral generic fibre and $\Delta(\pi) = 0$. Let $S$ be a finite set of places of $\QQ$
and $j \in \Z$.
Then the limit
\[
	\lim_{B\to \infty} 
	\frac
	{\#\{
	x\in \P^n(\Q):
	H(x)\leq B,
	\pi^{-1}(x) \text{ smooth},
	\omega_{\pi,S}(x)=j
	\}}
	{\#\{
	x\in \P^n(\Q):
	H(x)\leq B
	\}}
\]
exists and defines a probability measure on $\ZZ$.
\end{theorem} 

The analogues of the other results from \S \ref{sec:ps_intro} also hold for the
modified $\omega_{\pi,S}$. Also of course the analogue  of Theorem \ref{thm:gaussian} and the other results in \S \ref{sec:CLT}
trivially hold with $\omega_{\pi}$ replaced by $\omega_{\pi,S}$, since
$\omega_{\pi,S} = \omega_{\pi} + O(1)$.

\bibliographystyle{amsalpha}
\bibliography{normal_final}
\end{document}